\documentclass{amsart}[10pt]

\usepackage[czech, english]{babel} 

\usepackage{amsrefs}
\usepackage[all]{xy}
\usepackage{syntonly}
\usepackage{hyperref}
\usepackage{amsfonts}
\usepackage{amssymb}
\usepackage{amsmath}
\usepackage{amsthm}
\usepackage[inline]{enumitem}
\usepackage{tikz}
\usepackage{graphics}
\usepackage{graphicx}
\usepackage{color}
\usepackage{comment}
\usepackage{caption,lipsum}
\usepackage{framed}

\usepackage{lineno}

\usepackage{multirow}

\usepackage{footmisc}

\newtheorem{theorem}{Theorem}[section]
\newtheorem{corollary}[theorem]{Corollary}

\newtheorem{definition}[theorem]{Definition}
\newtheorem{lemma}[theorem]{Lemma}

\newtheorem{nonGlobalClaim}{Claim}[theorem]  

\newtheorem{observation}[theorem]{Observation}
\newtheorem{question}[theorem]{Question}

\newtheorem{fact}[theorem]{Fact}
\newtheorem{convention}[theorem]{Convention}
\newtheorem{remark}[theorem]{Remark}

\newtheorem{LC_Assumption}[theorem]{Large Cardinal Assumption}

\specialcomment{com}{ \color{red} \textbf{COMMENT:} }{  \color{black}} 
\specialcomment{oldcom}{ \color{brown} }{\color{black}} 

\excludecomment{oldcom} 

\excludecomment{com}

\usepackage{stmaryrd}

\usepackage[margin=1in]{geometry}

\begin{document}
\title[]{Maximum deconstructibility in module categories}

\author{Sean Cox}
\email{scox9@vcu.edu}
\address{
Department of Mathematics and Applied Mathematics \\
Virginia Commonwealth University \\
1015 Floyd Avenue \\
Richmond, Virginia 23284, USA 
}

\thanks{Many thanks to Jan \v{S}aroch for pointing out a significant error in a previous version of this paper.  The author is also grateful to Marco Aldi for helpful conversations about homological algebra, and to the anonymous referee for numerous improvements.}



\begin{abstract}
We prove that Vop\v{e}nka's Principle implies that for every class $\mathfrak{X}$ of modules over any ring, the class of \textbf{$\boldsymbol{\mathfrak{X}}$-Gorenstein Projective modules} (\textbf{$\boldsymbol{\mathfrak{X}}$-$\boldsymbol{\mathcal{GP}}$}) is a precovering class.  In particular, it is not possible to prove (unless Vop\v{e}nka's Principle is inconsistent) that there is a ring over which the \textbf{Ding Projectives} ($\boldsymbol{\mathcal{DP}}$) or the \textbf{Gorenstein Projectives} ($\boldsymbol{\mathcal{GP}}$) do not form a precovering class (\v{S}aroch previously obtained this result for the class $\mathcal{GP}$, using different methods).  The key innovation is a new ``top-down" characterization of \emph{deconstructibility}, which is a well-known sufficient condition for a class to be precovering.  We also prove that Vop\v{e}nka's Principle implies, in some sense, the maximum possible amount of deconstructibility in module categories.

\end{abstract}

\maketitle


\section{Introduction}\label{sec_Intro}

Given a ring $R$ (associative and unital), the class $\mathcal{P}_0$ of projective (left $R$-)modules plays a central role in classical homological algebra.  For a class $\mathcal{F}$ of modules, ``homological algebra relative to $\mathcal{F}$" attempts to employ similar methods, but with $\mathcal{F}$ playing the same role that $\mathcal{P}_0$ played in the classical setting.  For example, one uses $\mathcal{F}$-resolutions instead of projective resolutions, and attempts to define the $\text{Ext}^n$ functors ``relative to $\mathcal{F}$".  But in order for this to work---e.g., for the relative Ext functor to be independent of the $\mathcal{F}$-resolution---it is essential that $\mathcal{F}$ be a \emph{precovering} class (also called \emph{right-approximating} class).

Around the turn of the millennium, the notion of a \emph{deconstructible} class grew out of the solution of the Flat Cover Conjecture (Eklof-Trlifaj~\cite{MR1798574}; Bican-El Bashir-Enochs~\cite{MR1832549}).  Deconstructible classes are always precovering (\cite[Theorem 7.21]{MR2985554}), and showing that a class is deconstructible has become one of the main tools for showing the class is precovering.

We prove a new ``top-down" characterization of deconstructibility (Theorem \ref{thm_CharacterizeDecon}), and give two main applications, summarized by \ref{item_VP_X_GP} and \ref{item_VP_MaxDecon} below.  The applications are relative consistency results.  For example, \ref{item_VP_X_GP} implies that you won't be able to prove (in ZFC alone) that there exists a ring over which the Ding Projectives or the Gorenstein Projectives do not form precovering classes, unless \textbf{Vop\v{e}nka's Principle (VP)} is inconsistent.\footnote{Inconsistency of VP is generally regarded as highly unlikely; see the historical remarks in Chapter 6 of Ad\'{a}mek-Rosick\'{y}~\cite{MR1294136} for an amusing history of VP, which the authors describe as ``a practical joke which misfired".}

\begin{enumerate}[label=(\Roman*)]
  \item\label{item_VP_X_GP} VP implies that for \emph{any} class $\mathfrak{X}$ of modules (over any ring), the class $\boldsymbol{\mathfrak{X}\textbf{-}\mathcal{GP}}$ of ``$\mathfrak{X}$-Gorenstein Projective modules" is deconstructible, and hence precovering.  The general definition of $\mathfrak{X}$-$\mathcal{GP}$ appears in Section \ref{sec_GorenQuot}, but two particular cases are widely referenced in the literature:  the class of \textbf{Gorenstein Projectives} ($\boldsymbol{\mathcal{GP}}$) when $\mathfrak{X}=\{ \text{projectives} \}$ and the class of \textbf{Ding Projectives} ($\boldsymbol{\mathcal{DP}}$) when $\mathfrak{X} = \{ \text{flats} \}$.\footnote{The Ding Projectives were originally called \emph{Strong Gorenstein Flat} modules in Ding-Li-Mao~\cite{MR2529328}.}  Whether $\mathcal{GP}$, $\mathcal{DP}$, and/or other classes of the form $\mathfrak{X}$-$\mathcal{GP}$ are precovering,  has been asked or addressed repeatedly, often with partial results under various constraints on the ring; see, for example, \cite{MR2529328}, \cite{MR3621667}, \cite{MR1363858}, \cite{MR1753146}, \cite{MR3690524},  \cite{MR3598789}, \cite{MR3459032}, \cite{MR3760311}, \cite{MR2038564}, \cite{MR3839274}, \cite{MR4115324}, \cite{MR2283103},  \cite{MR2737778}, \cite{MR3473859}, \cite{MR4166732}.  The consistency result for the particular class $\mathcal{GP}$ was shown earlier by \v{S}aroch, via different methods that do not seem to generalize to arbitrary $\mathfrak{X}$-$\mathcal{GP}$'s; see the discussion after Theorem \ref{thm_Saroch}.

  \item\label{item_VP_MaxDecon} (``Maximum deconstructibility"):  VP implies that every class of modules, or of complexes of modules, that could conceivably be deconstructible, is in fact deconstructible.  See Lemma \ref{lem_QuotNecCondition} and Theorem \ref{thm_MainThm} for precisely what is meant by this.  
\end{enumerate}

We now describe the main results in more detail.  See Section \ref{sec_Prelim_ElemSub} for unexplained set-theoretic notation, Section \ref{sec_AlgRef} for the notion of $<\kappa$-Noetherian ring, and Section \ref{sec_ProofCharDecon} for the meaning of deconstructibility.

\begin{theorem}[characterization of deconstructibility]\label{thm_CharacterizeDecon}
Suppose $\kappa$ is a regular uncountable cardinal, $R$ is a $<\kappa$-Noetherian ring, and $\mathcal{F}$ is a class of $R$-modules.  Consider the following statements:

\begin{enumerate}[label=(\Roman*)$_\kappa$]
 \item\label{item_KappaDecon} $\mathcal{F}$ is strongly $<\kappa$-deconstructible.

 \item\label{item_ElemSub} Whenever $\mathcal{N}$ is an elementary submodel of $(H_\theta,\in,R,\kappa, \mathcal{F} \cap H_\theta )$ and $\mathcal{N} \cap \kappa $ is transitive---i.e., either $\mathcal{N} \cap \kappa \in \kappa$, or $\kappa \subset \mathcal{N}$---then for all $R$-modules $M$: 
 \begin{equation*}\label{eq_NandfracinF}
M \in \mathcal{N} \cap \mathcal{F} \ \implies  \langle \mathcal{N} \cap M \rangle \in \mathcal{F} \text{ and } \frac{M}{\langle \mathcal{N} \cap M \rangle} \in \mathcal{F}. 
 \end{equation*}

 \item\label{item_ElemSub_PAIRS} Whenever $\mathcal{N} \subseteq \mathcal{N}'$ are both elementary submodels of $(H_\theta,\in,R,\kappa, \mathcal{F} \cap H_\theta )$ and both $\mathcal{N} \cap \kappa$ and $\mathcal{N}' \cap \kappa$ are transitive, then for all $R$-modules $M$:
 
 \begin{equation*}
M \in \mathcal{N} \cap \mathcal{F} \ \implies \   \langle \mathcal{N} \cap M \rangle \in \mathcal{F} \text{ and } \frac{\langle \mathcal{N}' \cap M \rangle}{\langle \mathcal{N} \cap M \rangle} \in \mathcal{F}. 
 \end{equation*}

\end{enumerate}

\noindent The following implications always hold:
\begin{center}
\ref{item_KappaDecon}   $\Longleftarrow$  \ref{item_ElemSub} $\Longleftarrow$ \ref{item_ElemSub_PAIRS}.
\end{center}
If $\mathcal{F}$ is closed under transfinite extensions, then 
\begin{center}
\ref{item_KappaDecon}    $\iff$ \ref{item_ElemSub}  $\iff$ \ref{item_ElemSub_PAIRS}.
\end{center}
\end{theorem}

\noindent In particular, since every ring $R$ is (at worst) $<|R|^+$-Noetherian, we have:
\begin{corollary} 
If $\mathcal{F}$ is a class of modules closed under transfinite extensions, then $\mathcal{F}$ is deconstructible if and only if there exists a regular $\kappa$ such that \ref{item_ElemSub} holds.
\end{corollary}

\noindent Parts \ref{item_ElemSub} and \ref{item_ElemSub_PAIRS} could also be expressed in terms of an appropriate interpretation of Shelah's \emph{Stationary Logic}.  Though we will not make use of this connection, there is a brief discussion of it after Definition \ref{def_KappaAE}.  There is also a version of Theorem \ref{thm_CharacterizeDecon} for classes of \emph{complexes} of modules (Theorem \ref{thm_CharacterizeDecon_COMPLEXES} on page \pageref{thm_CharacterizeDecon_COMPLEXES}), though that takes a little more background to state.

We use Theorem \ref{thm_CharacterizeDecon} to isolate a necessary condition for deconstructibility, at least among classes that are closed under transfinite extensions:
\begin{lemma}[necessary condition for deconstructibility]\label{lem_QuotNecCondition}
If $\mathcal{F}$ is a class of modules (or complexes of modules) that is deconstructible and closed under transfinite extensions, then $\mathcal{F}$ is ``eventually almost everywhere closed under quotients" (Definition \ref{def_KappaAE} on page \pageref{def_KappaAE}).
\end{lemma}

Part \ref{item_MaxDecon} of the following theorem says that the necessary condition for deconstructibility given in Lemma \ref{lem_QuotNecCondition} is also sufficient, provided that Vop\v{e}nka's Principle holds.

\begin{theorem}\label{thm_MainThm}
Vop\v{e}nka's Principle implies:\footnote{The conclusions of Theorem \ref{thm_MainThm} are not first order statements, because they quantify over classes that do not have uniform definitions.  They are really schemes (as is VP itself).  See Remark \ref{rem_MetaMath} on page \pageref{rem_MetaMath} for exactly what is meant by this theorem.}  
\begin{enumerate}[label=(\Alph*)]
 \item\label{item_MaxDecon} (``Maximum deconstructibility") For all classes $\mathcal{K}$ of modules or complexes of modules over any ring:  if $\mathcal{K}$ is eventually almost everywhere closed under quotients and transfinite extensions, then $\mathcal{K}$ is deconstructible.
 \item\label{item_XGP_Decon} For all classes $\mathfrak{X}$ of modules over any ring, the class $\mathfrak{X}$-$\mathcal{GP}$ is deconstructible.
 \item\label{item_MaxKaplansky} For all classes $\mathcal{K}$ of modules or complexes of modules over any ring:  if $\mathcal{K}$ is eventually almost everywhere closed under quotients (but not necessarily under transfinite extensions), then $\mathcal{K}$ is a Kaplansky class.

\end{enumerate}
\end{theorem}

In the particular case $\mathfrak{X}=\{ \text{projectives} \}$---i.e., in the particular case of $\mathcal{GP}$---part \ref{item_XGP_Decon} of Theorem \ref{thm_MainThm} gives an alternate proof of the following recent unpublished theorem of \v{S}aroch, which partially solved the well-known problem of whether $\mathcal{GP}$ is always precovering:

\begin{theorem}[\v{S}aroch]\label{thm_Saroch}
If there are sufficiently many large cardinals, then the class $\mathcal{GP}$ of Gorenstein Projective modules is a deconstructible class over any ring.  
\end{theorem}

We will give our alternate proof of Theorem \ref{thm_Saroch}, before proving Theorem \ref{thm_MainThm}, because it still has all the main ideas, but makes use of large cardinals far weaker than Vop\v{e}nka's Principle.  We briefly compare \v{S}aroch's proof of Theorem \ref{thm_Saroch} with our proof:\footnote{In early November 2020, the author sent \v{S}aroch an alleged proof, in ZFC alone, that $\mathcal{GP}$ is deconstructible over any ring.  \v{S}aroch immediately pointed out a fatal error in the argument, and at the same time sent the author a preprint of his (\v{S}aroch's) earlier proof, via $\kappa$-pure-injectivity, that $\mathcal{GP}$ is always deconstructible if there is a proper class of strongly compact cardinals.  The author then realized that his own argument works \emph{if} one assumes a proper class of supercompact cardinals; more precisely, if the elementary submodels used in the problematic part of his earlier argument are assumed to be \emph{0-guessing} elementary submodels (see Definition \ref{def_0_guess} and Fact \ref{fact_Supercompact_0_guess}), then the error pointed out by \v{S}aroch is fixed.  This is the proof of Theorem \ref{thm_Saroch} given here, which is then generalized to prove Theorem \ref{thm_MainThm}.}
\begin{enumerate}[label=(\roman*)]
 \item\label{item_SarochsProof} \v{S}aroch's proof of Theorem \ref{thm_Saroch} assumes a proper class of strongly compact cardinals, and proceeds via the notion of $\kappa$-pure-injectivity.  The key to his proof is that if $\kappa$ is a strongly compact cardinal larger than the ring, then every projective module is $\kappa$-pure-injective; and that if every projective module is $\kappa$-pure-injective, then $\mathcal{GP}$ is $<\kappa$-deconstructible.
 
 \item\label{item_MyProof} Our proof of Theorem \ref{thm_Saroch}---which appears in Section \ref{sec_WithHelpOfLC}---uses a stronger large cardinal assumption, namely, a proper class of supercompact cardinals.  We show that if $\kappa$ is a supercompact cardinal larger than the ring, then the set of elementary submodels of $(H_\theta,\in)$ that are ``$\mathcal{GP}$-reflecting" (Definition \ref{def_K_reflect}) is stationary in $\wp_\kappa(H_\theta)$ for all $\theta > \kappa$; and that the stationarity of this set, together with some nice quotient behavior of the class $\mathcal{GP}$ with respect to $\mathcal{GP}$-reflecting elementary submodels, implies clause \ref{item_ElemSub} of Theorem \ref{thm_CharacterizeDecon} (when $\mathcal{F} = \mathcal{GP}$).
\end{enumerate}

\noindent The ideas in \ref{item_MyProof} are basically the same ones we use later to prove Theorem \ref{thm_MainThm}, which gets (under VP) deconstructibility of $\mathfrak{X}$-$\mathcal{GP}$ for all possible classes $\mathfrak{X}$.  \v{S}aroch's techniques described in \ref{item_SarochsProof} seem to generalize to those $\mathfrak{X}$ that can be shown (with the help of large cardinals) to be $\lambda$-pure-injective for some $\lambda$; however, it's not clear for which $\mathfrak{X}$ this could be done, even assuming VP.  In particular, it is not clear whether the techniques in \ref{item_SarochsProof} could be used to get deconstructibility of $\mathfrak{X}$-$\mathcal{GP}$ when $\mathfrak{X}$ is, say, the class of flat modules (i.e.\ to get deconstructibility of the Ding Projectives).  We note that for certain $\mathfrak{X}$, deconstructibility of $\mathfrak{X}$-$\mathcal{GP}$ is known to be a theorem of ZFC, e.g., when $\mathfrak{X}$ is closed under double duals and direct sums (\cite[Proposition 4.7]{MR4076700}) or the classes considered in \cite{bravo2014stable}.

Section \ref{sec_Prelim_ElemSub} provides the relevant background about elementary submodels of fragments of the universe of sets.  Section \ref{sec_AlgRef} proves a variety of lemmas with the following theme:  if $\mathcal{N}$ is a (possibly small) elementary submodel of the universe of sets, and $M$ is a module (or complex) such that $M$ is an element (but not necessarily a subset!) of $\mathcal{N}$, what properties of $M$ are inherited by $M \cap \mathcal{N}$?  This includes a variant of Kaplansky's Theorem (Lemma \ref{lem_Kaplansky}); a characterization of $<\kappa$-Noetherian rings in terms of elementary submodels (Theorem \ref{thm_CharactNoeth}); and the importance of the $<\kappa$-Noetherian property when restricting exact complexes to elementary submodels (Lemma \ref{lem_ExactRestrctNoProjective}).  Section \ref{sec_ProofCharDecon} proves Theorem \ref{thm_CharacterizeDecon} and states a variant for complexes of modules.  Section \ref{sec_AEclosureQuot} introduces the concept of eventual almost everywhere closure under quotients, and shows why classes of the form $\mathfrak{X}$-$\mathcal{GP}$ always have this property.  Section \ref{sec_K_ref_elem_subm} introduces, for a class $\mathcal{K}$ of complexes of modules the concept of a \textbf{$\boldsymbol{\mathcal{K}}$-reflecting elementary submodel} (Definition \ref{def_K_reflect}), and how stationarily many such models, in conjunction with eventual almost everywhere closure under quotients and transfinite extensions, guarantee deconstructibility of $\mathcal{K}$ and of certain associated classes of modules (Theorem \ref{thm_AbstractImplyDec}).   Section \ref{sec_WithHelpOfLC} proves Theorems \ref{thm_Saroch} and \ref{thm_MainThm};  the key use of the large cardinal in each of those theorems is to arrange that there are enough $\mathcal{K}$-reflecting elmentary submodels (for the relevent class $\mathcal{K}$ of complexes).  Section \ref{sec_Questions} lists some open questions, and Appendix \ref{sec_APPENDIX_VP} shows why Vop\v{e}nka's Principle implies the ad-hoc large cardinal property that was used in Section \ref{sec_WithHelpOfLC}.

\section{Preliminaries on elementary submodels}\label{sec_Prelim_ElemSub}

We briefly introduce the tools of \emph{elementary submodel} arguments; good accounts of such arguments appear in Soukup~\cite{MR2800978}, Dow~\cite{MR1031969}, and Geschke~\cite{MR1950041}.  The point of such arguments is that ``almost all" submodels of your (uncountable) structure of interest can be viewed as traces of structures which have a vast amount of set-theoretic closure.  The technique is analogous to viewing a function $f: \mathbb{R} \to \mathbb{R}$ as a trace of a complex function or surface, to exploit the tools of complex analysis to prove facts about the original function $f$.

All set-theoretic and model-theoretic terminology will follow Jech~\cite{MR1940513}, unless otherwise indicated.  If $L$ is a first order language and $\mathcal{A}$, $\mathcal{B}$ are $L$-structures, $\boldsymbol{\mathcal{A} \prec \mathcal{B}}$ means that $\mathcal{A}$ is an elementary submodel of $\mathcal{B}$; i.e., for all $L$-formulas $\phi(v_1, \dots, v_k)$ and all $a_1, \dots, a_k \in A$,
\[
\mathcal{A} \models \phi(a_1, \dots, a_k) \ \text{ if and only if } \ \mathcal{B} \models \phi(a_1,\dots,a_k),
\]
where $\models$ is the satisfaction relation.  If $V$ denotes the universe of sets and $L$ is the language of set theory---i.e., the language with a single binary predicate symbol $\dot{\in}$---we would often like to work with small (e.g.\ countable) elementary submodels of $(V,\in)$, and the reader won't lose much by just pretending all of our ``elementary submodels" are elementary in $(V,\in)$. However, by G\"odel's Second Incompleteness Theorem, there is no guarantee that such elementary submodels exist (at least not without large cardinal assumptions).  For this reason, set theorists often work with sufficiently closed initial segments of the universe; these initial segments are usually good enough for the intended application, and since these initial segments are sets (rather than proper classes), the Downward L\"owenheim-Skolem Theorem from first order logic can be applied to first order structures on them.  

 Cardinal numbers denoted by Greek letters are always assumed to be infinite.  Given a set $x$, the \textbf{transitive closure of $\boldsymbol{x}$}, denoted $\boldsymbol{\textbf{trcl}(x)}$, is the set of all $z$ such that there exists a finite sequence $u_0, \dots, u_n$ such that $z = u_n \in u_{n-1} \in \dots \in u_0 \in x$.  For a cardinal $\theta$, $\boldsymbol{H_\theta}$ denotes the set of all $x$ such that $|\text{trcl}(x)|<\theta$.  Every set is a member of some such $H_\theta$.  $H_\theta$ is a transitive set and, if $\theta$ is regular, $\boldsymbol{\mathcal{H}_\theta}:=(H_\theta,\in)$ models all axioms of ZFC except possibly the Powerset Axiom.  A useful fact is that $\mathcal{H}_\theta \prec_{\Sigma_1} (V,\in)$; i.e., if $\phi$ is a $\Sigma_1$ or $\Pi_1$ formula in the language of set theory, then for every finite list $p_1, \dots, p_k$ of parameters from $H_\theta$,
\[
(H_\theta,\in) \models \phi(p_1,\dots,p_k) \ \text{if and only if } (V,\in) \models \phi(p_1,\dots, p_k).
\]
Some examples of properties expressible by $\Sigma_1$ or $\Pi_1$ formulas that we will make use of are ``$M$ is a free $R$-module" and ``$M$ is a projective $R$-module".  Typically one doesn't even need the $\Sigma_1$-elementarity of $\mathcal{H}_\theta$ in $(V,\in)$; by choosing $\theta$ large enough to include all relevant parameters, statements of interest become $\Sigma_0$ statements,\footnote{I.e.\ all quantifiers are bounded by some parameter.  For example, ``$M$ is free" is $\Sigma_0$ in the parameters $\wp(M)$ and $R$.} which are absolute between $(V,\in)$ and $(H,\in)$ where $H$ is \emph{any} transitive set.  For most purposes, one could also use the rank initial segments of the universe (i.e., the $V_\alpha$'s) for such arguments.

We list some basic facts about elementary submodels of $\mathcal{H}_\theta$ (see the works cited earlier for more details).

\begin{fact}\label{fact_BasicElemSub}
Suppose $\mathcal{N} \prec \mathcal{H}_\theta$ (for some regular uncountable $\theta$).  Then $\mathcal{N}$  is closed under finite sequences and finite subsets, always contains $\omega$ and $\mathbb{Z}$ as elements and subsets, and has transitive intersection with $\omega_1$ (i.e., $\mathcal{N} \cap \omega_1 \in \omega_1 \cup \{ \omega_1 \}$).  If $f \in \mathcal{N}$ and $f$ is a function, then $\mathcal{N}$ is closed under $f$.\footnote{I.e., $f(x) \in \mathcal{N}$ whenever $x \in \mathcal{N} \cap \text{dom}(f)$}
If $\kappa$ is a cardinal and $\mathcal{N}$ has transitive intersection with $\kappa$,\footnote{For $\kappa > \omega_1$ this is not necessarily automatic, even for $\mathcal{N}$ of size $<\kappa$, due to the possibility that \emph{Chang's Conjecture} holds.} then for every $x \in \mathcal{N}$ of cardinality $<\kappa$, $x$ is a subset of $\mathcal{N}$.
\end{fact}

\noindent The last sentence of the previous fact is the reason that the phrase ``and $\mathcal{N} \cap \kappa$ is transitive" appears so often throughout this paper. For $\kappa = \omega_1$ this requirement is superfluous, since \emph{every} elementary submodel of some $\mathcal{H}_\theta$ has transitive intersection with $\omega_1$ (by the first part of the fact).

A sequence $\langle Z_\xi \ : \ \xi < \eta \rangle$ indexed by an ordinal $\eta$ will be called a \textbf{smooth $\boldsymbol{\in}$-chain} if it is $\subseteq$-increasing, $Z_\xi$ is both an element and subset of $Z_{\xi+1}$ whenever $\xi+1 < \eta$, and for all limit ordinals $\xi$, $Z_\xi = \bigcup_{\zeta < \xi} Z_\zeta$.  The following basic fact is a consequence of the Downward L\"owenheim-Skolem Theorem for first order logic.  
\begin{fact}\label{fact_SmoothChain}
Suppose $\kappa \le \lambda$ are both uncountable cardinals, and $\kappa$ is regular.  Let $\theta$ be a regular cardinal such that $\kappa,\lambda \in H_\theta$.  Let
\[
\mathfrak{A}  = (H_\theta,\in,\dots)
\]
be any first-order expansion of $(H_\theta,\in)$ in a countable signature.  Then there exists a smooth $\in$-chain
\[
\vec{\mathcal{N}} = \left\langle \mathcal{N}_\xi \ : \ \xi < \text{cf}(\lambda) \right\rangle
\]
such that for all $\xi < \text{cf}(\lambda)$:
\begin{enumerate}
 \item $\mathcal{N}_\xi \prec \mathfrak{A}$;

 \item $\mathcal{N}_\xi$ has transitive intersection with $\kappa$;
 \item $|\mathcal{N}_\xi|<\lambda$; and
 \item $\lambda \subseteq \bigcup_{\xi < \text{cf}(\lambda)} \mathcal{N}_\xi$. 
\end{enumerate}
\end{fact}
\begin{proof}
Fix an increasing sequence $s=\langle s_i \ : \ i < \text{cf}(\lambda) \rangle$ of ordinals that is cofinal in $\lambda$.  If $\kappa = \lambda$, use the Downward L\"owenheim-Skolem Theorem (and regularity of $\kappa$) to recursively build a $\kappa$-length, smooth $\in$-chain of $<\kappa$-sized elementary submodels of $\mathfrak{A}$ whose intersections with $\kappa$ are elements of $\kappa$, and such that the $i+1$-st structure's intersection with $\kappa$ is at least $s_i$.  For $\lambda > \kappa$ ($\lambda$ possibly singular), use the Downward L\"owenheim-Skolem Theorem, and regularity of $\text{cf}(\lambda)$, to recursively define a smooth $\text{cf}(\lambda)$-length $\in$-chain of $<\lambda$-sized elementary submodels of $\mathfrak{A}$, such that the $i+1$-st submodel contains $s_i$ as a subset; since $\lambda > \kappa$ we can without loss of generality assume $\kappa \le s_0$, so each submodel contains $\kappa$ as a subset.
\end{proof}

We will also need the following lemma at one point in the proof of Theorem \ref{thm_AbstractImplyDec}.  It could also be expressed in terms of a certain interpretation of \emph{Stationary Logic}, but the following characterization will be more directly helpful for our needs.\footnote{If $\kappa = \aleph_1$, \ref{item_ElemSub} is equivalent to asserting there is an $F:[M]^{<\omega} \to M$ such that whenever $N \subset M$ is closed under $F$ ($N$ of any cardinality), then both $\langle N \rangle$ and $M/\langle N \rangle$ are in $\mathcal{F}$. }  ZF$^-$ denotes the Zermelo-Fraenkel axioms of set theory, minus the powerset axiom.

\begin{lemma}\label{lem_NewAlgebraLemma}
Given any ring $R$ (not necessarily $<\kappa$-Noetherian), the following are equivalent:
 \begin{enumerate}
  \item Statement \ref{item_ElemSub} of Theorem \ref{thm_CharacterizeDecon};
  \item\label{item_ExistsTransH} For every $M \in \mathcal{F}$, there exists a pair $(H,\mathfrak{A})$ such that $(H,\in)$ is a transitive $\text{ZF}^-$ model, $\{\kappa, R,M \} \subset H$, and $\mathfrak{A}$ is an expansion of $(H,\in)$ in a countable first order signature such that whenever $M \in \mathcal{N} \prec \mathfrak{A}$ and $\mathcal{N} \cap \kappa$ is transitive, then $\langle \mathcal{N} \cap M \rangle \in \mathcal{F}$ and $M/\langle \mathcal{N} \cap M \rangle \in \mathcal{F}$.
\end{enumerate}
\end{lemma}

\begin{proof}
The $\Rightarrow$ direction is trivial, since the $H$ is witnessed by any $H_\theta$ such that $M \in H_\theta$, and the $\mathfrak{A}$ is witnessed by 
\[
(H_\theta,\in,R,\kappa,\mathcal{F} \cap H_\theta).
\]

Now suppose \eqref{item_ExistsTransH} holds.  First we claim:
\begin{nonGlobalClaim}\label{clm_CanTakeHsmall}
For any $M \in \mathcal{F}$, there is an $H$ as in \eqref{item_ExistsTransH} such that 
\[
|H|= \text{max}\Big( \kappa,|R|,|\text{trcl}(M)|  \Big) =: \lambda_M.
\]
\end{nonGlobalClaim}
\begin{proof}
(of Claim \ref{clm_CanTakeHsmall}): By the Downward L\"owenheim-Skolem Theorem, there is an $\mathfrak{X}=(X,\in,\dots) \prec \mathfrak{A}$ of cardinality at most $\lambda_M$ such that $\kappa \cup \text{trcl}(M) \subseteq X$.  The structure $(H,\in)$ is a transitive $\text{ZF}^-$ model, so in particular $X$ is extensional.  Then there is a transitive set $H_X$ and an isomorphism $\sigma: (H_X,\in) \to (X,\in)$.  Let $\mathfrak{A}_X$ be the result of transferring the structure $\mathfrak{X}$ to $H_X$.  We claim that $H_X$ and $\mathfrak{A}_X$ still witness the requirements of part \eqref{item_ExistsTransH}.  Consider any $\mathcal{N} \prec \mathfrak{A}_X$ such that $\mathcal{N} \cap \kappa$ is transitive.  Let $\mathcal{N}'$ be the pointwise image of $\mathcal{N}$ under $\sigma$; then $\mathcal{N}' \prec \mathfrak{X} \prec \mathfrak{A}$.  Since $\kappa \cup \text{trcl}(M) \subset X$, $\sigma$ fixes $M$ and $\kappa$, so 
\[
\mathcal{N}' \cap \kappa = \mathcal{N} \cap \kappa \text{ and } \mathcal{N}' \cap M = \mathcal{N} \cap M =:Z.
\]
Then by assumption about $\mathfrak{A}$, it follows that both $\langle Z \rangle$ and $M/\langle Z \rangle$ are in $\mathcal{F}$.  
\end{proof}

So we can without loss of generality assume that the $H$ in part \eqref{item_ExistsTransH} is always of size at most $\lambda_M$.  Now consider any regular $\theta$ such that $R,\kappa \in H_\theta$.  Let $\mathcal{F}_\theta:= \mathcal{F} \cap H_\theta$, and consider an arbitrary 
\[
\mathcal{Q} \prec (H_\theta,\in,R,\kappa,\mathcal{F}_\theta).
\]
such that $\mathcal{Q} \cap \kappa$ is transitive.  Fix any $M \in \mathcal{Q} \cap \mathcal{F} = \mathcal{Q} \cap \mathcal{F}_\theta$.  Since $\wp(x) \subset H_\theta$ for every $x \in H_\theta$, the following assertion (about the parameters $M$,$R$,$\kappa$) is downward absolute from $(V,\in)$ to the structure $(H_\theta,\in,R,\kappa,\mathcal{F}_\theta)$ (viewing $\mathcal{F}_\theta$ as a predicate):
\begin{quote}
``there is a transitive $\text{ZF}^-$ model $H$ of size $\text{max}\Big( \kappa,|R|,|\text{trcl}(M)| \Big)$ with $\{ R,M,\kappa \} \subset H$ and some $\mathfrak{A}$ extending $(H,\in)$ in a countable signature such that whenever $\mathcal{N} \prec \mathfrak{A}$ and $\mathcal{N} \cap \kappa$ is transitive, then both $\langle \mathcal{N} \cap M \rangle$ and $M/\langle \mathcal{N} \cap M \rangle$ are in $\mathcal{F}_\theta$" 
\end{quote}
Since $M \in \mathcal{Q}$, there is such an $H$ and $\mathfrak{A}$ that are \emph{elements} of $\mathcal{Q}$.  Since $\mathfrak{A} \in \mathcal{Q}$ and the signature of $\mathfrak{A}$ is countable---in particular contained as a subset of $\mathcal{Q}$---it follows that $\mathcal{Q}_0:=\mathcal{Q} \cap H \prec \mathfrak{A}$, and since $\kappa \subset H$, $\mathcal{Q}_0 \cap \kappa = \mathcal{Q} \cap \kappa$ is transitive.  So $\langle \mathcal{Q}_0 \cap M \rangle$ and $M/\langle \mathcal{Q}_0 \cap M \rangle$ are both in $\mathcal{F}_\theta$.  But since $M$ is also a subset of $H$, $\langle \mathcal{Q}_0 \cap M \rangle = \langle \mathcal{Q} \cap M \rangle$.  This completes the proof.

\end{proof}

\section{Elementary submodels and basic algebraic reflection}\label{sec_AlgRef}

By \emph{module} and \emph{ideal} we will officially mean \emph{left module} and \emph{left ideal}, respectively.  If $M$ is an $R$-module and $X \subseteq M$, $\boldsymbol{\langle X \rangle^M_R}$ denotes the $R$-submodule of $M$ generated by $X$, though we will just write $\langle X \rangle$ when the ring and ambient module are clear from the context.  This section is mainly about the following kind of question:

\begin{question}
Suppose $R$ is a ring, $M$ is an $R$-module, and $\mathcal{N}$ is an elementary submodel of some $\mathcal{H}_\theta$ such that $R$ and $M$ are both elements of---but not necessarily subsets of---the structure $\mathcal{N}$.  Note that by elementarity of $\mathcal{N}$, $\mathcal{N} \cap M$ is always closed under addition, and is in fact an $\mathcal{N} \cap R$-module.  What properties of $M$ also hold of $\langle \mathcal{N} \cap M \rangle^M_R$?   What about $M/\langle \mathcal{N} \cap M \rangle^M_R$?   If $M_\bullet \in \mathcal{N}$ is a complex of modules, what properties of $M_\bullet$ hold of its restriction to $\mathcal{N}$?
\end{question}

We will see:
\begin{itemize}
 \item  Section \ref{sec_Kaplansky} shows that \emph{projectivity} behaves very nicely with respect to elementary submodels.  This gives an alternative way to prove the classic theorem of Kaplansky, but will also be used in the proof of Theorems \ref{thm_Saroch} and \ref{thm_MainThm} (specifically, in Lemma \ref{lem_GorRefLemma}).
 \item Section \ref{sec_Noetherian} shows that if $\mathcal{N}$ has transitive intersection with the Noetherian degree of the ring, and $M_\bullet$ is an exact complex of modules with $M_\bullet \in \mathcal{N}$, then the restriction of $M_\bullet$ to $\mathcal{N}$ remains exact.  
\end{itemize}

\subsection{Kaplansky's Theorem revisited}\label{sec_Kaplansky}

Although submodules of projective modules need not be projective, Lemma \ref{lem_Kaplansky} basically says that---even for non-hereditary rings---``almost all" submodules of projective modules are projective, and moreover have projective quotients.  Lemma \ref{lem_Kaplansky} can be used in conjunction with (a variant of) Theorem \ref{thm_CharacterizeDecon} to give an alternative proof of Kaplansky's Theorem (and recent variants, such as \cite{MR4115358}).

First, recall (\cite{MR1653294}) that an $R$-module $P$ is projective if and only if it has a \textbf{dual basis}, which is a pair
\[
\mathcal{D}=\Big(  B, \big( f_b \big)_{b \in B} \Big)
\]
such that $B \subseteq P$,\footnote{Dual bases are typically defined using an index set $I$ and a function $i \mapsto (b_i,f_i)$ for each $i \in I$, which is not required to be injective.  However, the definition given here also characterizes projectivity; in fact, if $P$ is projective, one can find a dual basis with $i \mapsto b_i$ injective, with range exactly $P$. } each $f_b: P \to R$ is an $R$-linear map, and for every $x \in P$,
\[
\text{sprt}_{\mathcal{D}}(x):=  \{ b \in B \ : \ f_b(x) \ne 0 \} 
\]
is finite, and 
\[
x = \sum_{b \in \text{sprt}_{\mathcal{D}}(x)} f_b(x) b.
\]
If $\mathcal{D}$ is a dual basis for $P$ and $M$ is an $R$-submodule of $P$, let us say that \textbf{$\boldsymbol{M}$ is closed under $\boldsymbol{\mathcal{D}}$-supports} if $\text{sprt}_{\mathcal{D}}(x) \subseteq M$ for every $x \in M$.

\begin{lemma}\label{lem_SupportClosureProjQuot}
If $M$ is closed under $\mathcal{D}$-supports, then both $M$ and $P/M$ are projective.
\end{lemma}
\begin{proof}
The obvious restriction of $\mathcal{D}$ to $M$ is a dual basis for $M$.  To get a dual basis for $P/M$, first note that if $b \in B \setminus M$ and $x + M = y + M$, then $f_b(x) = f_b(y)$; otherwise, $f_b(x-y)$ would be nonzero, and since $z:=x-y \in M$ and $M$ is closed under $\mathcal{D}$-supports, $b$ would be an element of $M$, a contradiction.  So for $b \in B \setminus M$, the function $\hat{f}_b: P/M \to R$ defined by $x + M \mapsto f_b(x)$ is well-defined.  This yields a dual basis for $P/M$ (indexed by $\{ b + M \ : \ b \in B \setminus M \}$).  
\end{proof}

\begin{lemma}\label{lem_Kaplansky} If $P$ is a projective $R$-module with $R,P \in \mathcal{N} \prec \mathcal{N}' \prec \mathcal{H}_\theta$, then 
 \[
\langle \mathcal{N} \cap P \rangle, \ \frac{\langle \mathcal{N}' \cap P \rangle}{\langle \mathcal{N} \cap P \rangle}, \text{ and } \frac{P}{\langle \mathcal{N} \cap P \rangle}
\]
are each projective. 
\end{lemma}

\begin{proof}
We could prove this using the direct sum decomposition of $P$ given by Kaplansky's Theorem,\footnote{That every projective module is a direct sum of countably generated projective modules.} but we prefer to give a direct proof, in order to show the power of elementary submodels in this context.  Since $P$ has a dual basis in $V$, the $\Sigma_1$-elementarity of $\mathcal{H}_\theta$ in $(V,\in)$ yields a $\mathcal{D} \in H_\theta$ such that $\mathcal{H}_\theta \models$ ``$\mathcal{D}$ is a dual basis for $P$".  Then by elementarity of $\mathcal{N}$ in $\mathcal{H}_\theta$, we can without loss of generality assume $\mathcal{D} \in \mathcal{N}$.  We claim that $\langle \mathcal{N} \cap P \rangle$ is closed under $\mathcal{D}$-supports; once we have this, the same is obviously true of $\mathcal{N}'$ too, and then (by Lemma \ref{lem_SupportClosureProjQuot}) one gets the projectivity of each of the three modules mentioned in the statement of the current lemma.

Suppose $x \in \langle \mathcal{N} \cap P \rangle$; then $x = \sum_{k=1}^n r_k z_k$ for some $n \in \mathbb{N}$, some $r_k \in R$ and $z_k \in \mathcal{N} \cap P$.  Then for any $b \in B$, $f_b(x) = \sum_{k=1}^n r_k f_b(z_k)$, which implies
 \begin{equation}\label{eq_SprtContained}
\text{sprt}_{\mathcal{D}}(x) \subseteq \bigcup_{k=1}^n \text{sprt}_{\mathcal{D}} (z_k).
\end{equation}

\noindent Since each $z_k \in \mathcal{N}$ and $\mathcal{D} \in \mathcal{N}$, $\text{sprt}_{\mathcal{D}}(z_k)$ is an element of $\mathcal{N}$ for each $k=1,\dots,n$; and since it is finite, it is a subset of $\mathcal{N}$ by Fact \ref{fact_BasicElemSub}.  Also by Fact \ref{fact_BasicElemSub}, $\mathcal{N}$ is closed under finite unions, so the entire right side of \eqref{eq_SprtContained} is contained--both as an element and as a subset--in $\mathcal{N}$.
\end{proof}

\begin{corollary}\label{cor_WitnessNonExact}
Suppose $A$, $B$, and $C$ are $R$-modules, and $B$ is $\mu$-generated, where $\mu$ is an infinite cardinal. If $\xymatrix{A \ar[r]^f & B \ar[r]^g & C}$ is exact at $B$ but there is some projective $Q$ such that 
\[
\xymatrix{\text{Hom}_R(C,Q) \ar[r]^{g^*} & \text{Hom}_R(B,Q) \ar[r]^{f^*} & \text{Hom}_R(A,Q)}\]
is not exact at $\text{Hom}_R(B,Q)$, then there is such a $Q$ that is $\mu$-generated. 

It follows that if $M_\bullet = \Big( \xymatrix{M_n \ar[r]^{f_n} & M_{n+1}}  \Big)_{n \in \mathbb{Z}}$ is a complex of $R$-modules, $M_\bullet$ and $R$ are elements of $H_\lambda$ where $\lambda$ is an uncountable cardinal, then the statement
\begin{quote}
``$M_\bullet$ is $\text{Hom}_R(-,\text{Proj})$-exact"
\end{quote} 
is absolute between $\mathcal{H}_\lambda$ and the universe $V$ of sets.
\end{corollary}
\begin{proof}
Let $Q$ be projective and $\sigma \in \text{ker}\big(f^* \big)$ witness non-exactness of the Hom sequence at the middle term.  Clearly, to prove the corollary, it suffices to find a $\mu$-generated, projective submodule of $Q$ that contains the image of $\sigma$.  Let $Z$ be a $\mu$-sized generating set for $B$; then $Z':=\sigma[Z]$ is a $\le \mu$-sized generating set for the image of $\sigma$; let $Y$ denote this image.  Fix any $\theta$ such that $Z', R, Q,Y \in H_\theta$; by the Downward L\"owenheim-Skolem Theorem, there is an $\mathcal{N}$ such that
\[
Z' \cup \{ Z',R,Q,Y \}  \subset \mathcal{N} \prec \mathcal{H}_\theta \text{ and } |\mathcal{N}| = \mu.
\]
Then $\langle \mathcal{N} \cap Q \rangle^Q$ is $\mu$-generated, is projective by Lemma \ref{lem_Kaplansky}, and contains $Y$  because $\mathcal{N} \supset Z'$.

Now we prove the second part of the corollary.  For the downward absoluteness, suppose $\mathcal{H}_\lambda \models$ ``$M_\bullet$ is \textbf{not} $\text{Hom}_R(-,\text{Proj})$-exact"; this is witnessed by an offending $\sigma: M_n \to Q$ in $H_\lambda$ where $\sigma,Q \in H_\lambda$ and $Q$ is projective from the point of view of $\mathcal{H}_\lambda$; this is easily upward absolute to $V$ (recall projectivity is $\Sigma_1$).  

For the other direction, suppose $V \models$ ``$M_\bullet$ is \textbf{not} $\text{Hom}_R(-,\text{Proj})$-exact".  By the first part of the corollary, this is witnessed by some offending $\sigma: M_n \to Q$ where $Q$ is projective and $|M_n|$-generated.  Then $|Q| \le |R|\cdot|M_n|$, and since both $M_n$ and $R$ are elements of $H_\lambda$, $|R|\cdot|M_n| < \lambda$.  So both $\sigma$ and $Q$ are (without loss of generality) elements of $H_\lambda$.  Since $\mathcal{H}_\lambda \prec_{\Sigma_1} V$, $Q$ is projective from the point of view of $\mathcal{H}_\lambda$ too (and clearly $\sigma$ is not in the range of $f_{n+1}^*$ from the point of view of $\mathcal{H}_\lambda$ either).  
\end{proof}

\subsection{$<\kappa$-Noetherian rings}\label{sec_Noetherian}

If $M_\bullet$ is a sequence of homomorphisms and $M_\bullet \in \mathcal{N}$, one can consider ``restricting" $M_\bullet$ to the submodules generated by $\mathcal{N}$ itself.  The next lemma provides very basic facts about this procedure, and the subsequent remark motivates why $<\kappa$-Noetherian rings come into play.

\begin{lemma}\label{lem_restrictComplexN}
Suppose
\[
\xymatrix{
M_\bullet: & \dots \ar[r] &  M_{n-1} \ar[r]^-{f_{n-1}} & M_{n} \ar[r]^-{f_{n}} & M_{n+1} \ar[r] & \dots \ (n \in \mathbb{Z})
}
\]
is a sequence of $R$-module homomorphisms indexed by $\mathbb{Z}$.  Suppose that $R$ and $M_\bullet$ are elements of $\mathcal{N}$, and $\mathcal{N} \prec \mathcal{H}_\theta$.  Then for each $n \in \mathbb{Z}$, $f_{n-1}$ is an element of $\mathcal{N}$, and $f_{n-1} \restriction \langle \mathcal{N} \cap M_{n-1} \rangle$ maps into $\langle \mathcal{N} \cap M_{n} \rangle$.  Furthermore,
\begin{equation}\label{eq_ImAt_n_NoNoeth}
 \text{im}\Big( f_{n-1} \restriction \langle \mathcal{N} \cap M_{n-1} \rangle \Big) \subseteq  \text{im}(f_{n-1}) \cap \langle \mathcal{N} \cap M_n \rangle
\end{equation}
and
\begin{equation}\label{eq_KerAt_n_NoNoeth}
\text{ker}\Big( f_n \restriction \langle \mathcal{N} \cap M_n \rangle \Big) = \text{ker}(f_n) \cap \langle \mathcal{N} \cap M_n \rangle.
\end{equation}

\noindent \textbf{Notation:} We will denote the sequence 
\[
\xymatrix{
\dots  \langle \mathcal{N} \cap  M_n \rangle \ar[rr]^-{f_{n} \restriction \langle \mathcal{N} \cap  M_n \rangle} && \langle \mathcal{N} \cap M_{n+1} \rangle \ar[rr]^-{f_{n+1} \restriction \langle \mathcal{N} \cap  M_{n+1} \rangle} && \langle \mathcal{N} \cap  M_{n+2} \rangle  \dots 
}
\]
by $\boldsymbol{M_\bullet \restriction \mathcal{N}}$, and we will denote the quotient $M_\bullet / \big( M_\bullet \restriction \mathcal{N}\big)$ by $\boldsymbol{M_\bullet / \mathcal{N}}$.  The quotient map $M_\bullet \to M_\bullet / \mathcal{N}$ will be denoted by $\pi_\bullet$.
\end{lemma}
\begin{proof}
Since $M_\bullet$ is a sequence with domain $\mathbb{Z}$, $\mathbb{Z}$ is an element and subset of $\mathcal{N}$ by Fact \ref{fact_BasicElemSub}, and $M_\bullet \in \mathcal{N}$, it follows by elementarity of $\mathcal{N}$ in $\mathcal{H}_\theta$ that each $f_n$ and each $M_n$ is an element of $\mathcal{N}$.  Elementarity of $\mathcal{N}$ then ensures that $f_{n-1} \restriction (\mathcal{N} \cap M_{n-1})$ maps into $\langle \mathcal{N} \cap M_{n} \rangle$.  Then the $R$-linearity of $f_{n-1}$ ensures that it also maps $\langle \mathcal{N} \cap M_{n-1} \rangle$ into $\langle \mathcal{N} \cap M_{n} \rangle$, yielding the inclusion \eqref{eq_ImAt_n_NoNoeth}.  The equality \eqref{eq_KerAt_n_NoNoeth} is obvious.  
\end{proof}

\begin{remark}\label{rem_RestrictExact}
Suppose $M_\bullet$ is a complex of $R$-modules, and $R$ and $M_\bullet$ are elements of $\mathcal{N}$, where $\mathcal{N} \prec \mathcal{H}_\theta$.  Then obviously $M_\bullet \restriction \mathcal{N}$ is also a complex, since any restriction of a complex is another complex.  However, if $M_\bullet$ is exact, we seem to need some further assumption to ensure that $M_\bullet \restriction \mathcal{N}$ will be exact.  To see where the problem arises, consider the problem of showing that the kernel of $f_{n+1} \restriction \langle \mathcal{N} \cap M_{n+1} \rangle$ is contained in the image of $f_{n} \restriction \langle \mathcal{N} \cap M_{n}\rangle$.  Say $y=\sum_i r_i b_i$ is in the kernel of $f_{n+1} \restriction \langle \mathcal{N} \cap M_{n+1} \rangle$, where each $b_i$ is an element of $\mathcal{N} \cap M_{n+1}$ and each $r_i$ is an element of $R$.  Unless we know that each $b_i$ itself is in the kernel of $f_{n+1}$, it is not clear how to show that $y$ is in the image of $f_{n} \restriction \langle \mathcal{N} \cap M_{n} \rangle$.  We seem to need something like the following equality, where $K:= \text{ker}(f_{n+1}) $:\footnote{Note that $K \in \mathcal{N}$ because it is definable from $f_{n+1}$, and $f_{n+1}$ is an element of $\mathcal{N}$.}
\begin{equation}\label{eq_WantAtK}
K \cap \langle \mathcal{N} \cap M_{n+1} \rangle = \langle \mathcal{N} \cap K \rangle.
\end{equation}
The $\supseteq$ direction of \eqref{eq_WantAtK} is trivially true, but the other inclusion is false in general (see Theorem \ref{thm_CharactNoeth} below).  If $R \subset \mathcal{N}$ then the entire problem trivializes, since $\langle \mathcal{N} \cap M_{n+1} \rangle = \mathcal{N} \cap M_{n+1}$ and $\langle \mathcal{N} \cap K \rangle = \mathcal{N} \cap K$ in that case.  But, as discussed below, we would like to avoid having to assume that $R \subset \mathcal{N}$.  We will see in Section \ref{sec_Noetherian} that if $R$ is $<\kappa$-Noetherian and  $\mathcal{N} \cap \kappa$ is transitive, then the equality \eqref{eq_WantAtK} will indeed hold.

\end{remark}

As noted in Remark \ref{rem_RestrictExact}, one way to ensure that exactness of $M_\bullet$ implies exactness of $M_\bullet \restriction \mathcal{N}$ (assuming $M_\bullet \in \mathcal{N}$) is to assume that $R \subset \mathcal{N}$.
The reader who is only interested in the general results of the kind ``such-and-such class is deconstructible", without caring about the degree of the deconstructibility, may as well just assume that all of the elementary submodels considered have $R$ as a subset, and skip the results below about $<\kappa$-Noetherian rings.  However, if one is interested in showing that a class is, say, $<\aleph_1$-deconstructible but the ring is uncountable---e.g.\ if one wants to address the (still open) problem of whether $\mathcal{GP}$ is $<\aleph_1$-deconstructible for all $<\aleph_1$-Noetherian rings $R$---then one will likely need to deal with countable elementary submodels that do \textbf{not} contain $R$ as a subset (see Theorem \ref{thm_AbstractImplyDec} for how one might attempt to do this).

 A ring is \textbf{$\boldsymbol{<\kappa}$-Noetherian} if all ideals in $R$ are $<\kappa$-generated; so $<\aleph_0$-Noetherian is the ordinary Noetherian property, and $R$ is always (at worst) $<|R|^+$-Noetherian.  We write $\aleph_0$-Noetherian instead of $<\aleph_1$-Noetherian.  Theorem \ref{thm_CharactNoeth} below characterizes the $<\kappa$-Noetherian rings.  It is motivated by the problem pointed out in Remark \ref{rem_RestrictExact}.
First we need a lemma:
\begin{lemma}\label{lem_SmallGenSet}, 
Suppose $M$ is a $<\kappa$-generated $R$-module, and $\mathcal{N} \prec \mathcal{H}_\theta$ is such that $M,R \in \mathcal{N}$ and $\mathcal{N} \cap \kappa$ is transitive.  Then $\mathcal{N} \cap M$ generates $M$.
\end{lemma}
\begin{proof}
By elementarity of $\mathcal{N}$, there is an $X \in \mathcal{N}$ of size $<\kappa$ that generates $M$.  By Fact \ref{fact_BasicElemSub} and the assumption that $\mathcal{N} \cap \kappa$ is transitive, $X \subset \mathcal{N}$.  Hence $M=\langle X \rangle = \langle \mathcal{N} \cap M \rangle$.
\end{proof}

\begin{theorem}\label{thm_CharactNoeth}
Let $\kappa$ be an infinite regular cardinal.  Let $R$ be an associative ring (but possibly non-unital).  Consider the following assertions about $R$:
\begin{enumerate}[label=(\alph*)]
 \item\label{item_Noeth} $R$ is $<\kappa$-Noetherian.
 \item\label{item_AllElemSubm} Whenever $G$ is an $R$-submodule of $M$, and $\mathcal{N}$ is an elementary submodel of some $\mathcal{H}_\theta$ such that:
\begin{itemize} 
 \item $\mathcal{N} \cap \kappa$ is transitive; and
 \item $R$, $G$, and $M$ are elements of $\mathcal{N}$;
\end{itemize}
 then
\begin{equation}\label{eq_Noeth}
G \cap \langle \mathcal{N} \cap M \rangle^M_R = \langle \mathcal{N} \cap G \rangle^M_R.
\end{equation}
\end{enumerate}

\noindent Statement \ref{item_Noeth} always implies statement \ref{item_AllElemSubm}.  If $\kappa$ is uncountable and $R$ is unital, then the converse is also true. 
\end{theorem}
\begin{proof}

To prove \ref{item_Noeth} $\implies$ \ref{item_AllElemSubm}, assume that $R$ is $<\kappa$-Noetherian, and let $G$, $M$, and $\mathcal{N}$ be as in the hypothesis of part \ref{item_AllElemSubm}.  The $\supseteq$ direction of \eqref{eq_Noeth} is trivially true, so we just need to show the $\subseteq$ direction; we first show it under the additional assumption that $M$ is free, and will denote $M$ by $F$ for this special case.  By elementarity of $\mathcal{N}$ there is a basis $\vec{b} = \{ b_i \ : \ i < \mu \}$ of $F$ such that $\vec{b} \in \mathcal{N}$.  For each $\alpha \le \mu$, define
\[
G_{<\alpha}:= G \cap \left\langle b_i \ : \ i < \alpha  \right\rangle   
\]
and for each $\alpha < \mu$ set 
\[
 G_{\le\alpha}:= G \cap \left\langle b_i \ : \ i \le \alpha  \right\rangle
\]
and define
\[
\phi_\alpha: G_{\le \alpha} \to R
\]
by extracting the coefficient on $b_\alpha$.  Then $\text{ker}(\phi_\alpha) = G_{<\alpha}$, so $G_{\le \alpha} / G_{<\alpha}$ is isomorphic to an ideal in $R$, and hence by the $<\kappa$-Noetherian property,
\begin{equation}
G_{\le \alpha} / G_{<\alpha} \text{ is } <\kappa \text{-generated.}
\end{equation}

If $\alpha$ happens to be an element of $\mathcal{N}$, then $G_{<\alpha}$ and $G_{\le \alpha}$ are both elements of $\mathcal{N}$, so $\mathcal{N}$ sees that $G_{\le \alpha} / G_{<\alpha}$ is $<\kappa$-generated.   So, by Lemma \ref{lem_SmallGenSet}, 
\begin{equation}\label{eq_QuotientInN_2}
\forall \alpha \in \mathcal{N} \cap \mu \ \ \  \mathcal{N} \cap \frac{G_{\le \alpha}}{G_{<\alpha}} \text{ generates } \frac{G_{\le \alpha}}{G_{<\alpha}}.
\end{equation}

The following claim will finish the free case (note that $\mu \in \mathcal{N}$ because it is the length of $\vec{b}$, and $\vec{b} \in \mathcal{N}$):
\begin{nonGlobalClaim}\label{clm_NcapGlessalpha}
For all $\alpha \in \mathcal{N} \cap [0,\mu]$, 
\[
\mathcal{N} \cap G_{<\alpha} \text{ generates } \langle \mathcal{N} \rangle^F \cap G_{<\alpha}.
\] 
\end{nonGlobalClaim}
\begin{proof}
(of Claim \ref{clm_NcapGlessalpha}):  Suppose $\alpha \in \mathcal{N} \cap [0,\mu]$ and the statement holds at all $\beta \in \mathcal{N} \cap [0,\alpha)$.  If $\alpha$ is a limit ordinal, then (by elementarity of $\mathcal{N}$) there is no largest element of $\mathcal{N} \cap \alpha$; it follows that any element of  
\[
\langle \mathcal{N} \rangle^F \cap G_{<\alpha}
\]
is an element of $\langle \mathcal{N} \rangle^F \cap G_{<\beta}$ for some $\beta \in \mathcal{N} \cap \alpha$, and hence a linear combination of members of $\mathcal{N} \cap G_{<\beta}$ by the induction hypothesis.

Now suppose $\alpha \in \mathcal{N}$ is a successor ordinal, say, $\alpha = \beta+1$.  Then $\beta \in \mathcal{N}$ too.  By the induction hypothesis,
\[
\mathcal{N} \cap G_{<\beta} \text{ generates }  \langle \mathcal{N}\rangle^F \cap G_{<\beta}. 
\]

And by \eqref{eq_QuotientInN_2},
\begin{equation}\label{eq_QuotBeta}
\mathcal{N} \cap \frac{G_{\le \beta}}{G_{< \beta}} \text{ generates } \frac{G_{\le \beta}}{G_{< \beta}}
\end{equation}

We need to show that $\mathcal{N} \cap G_{\le \beta}$ generates $\langle \mathcal{N}\rangle^F \cap G_{\le \beta}$; so assume $g \in \langle \mathcal{N}\rangle^F \cap G_{\le \beta}$.  By \eqref{eq_QuotBeta}, there are cosets $c_1, \dots, c_k$, each in $\mathcal{N} \cap \frac{G_{\le \beta}}{G_{<\beta}}$, and some $r_1,\dots r_k \in R$ such that $\sum_{m=1}^k r_m c_m = g + G_{<\beta}$.  By elementarity of $\mathcal{N}$, $c_m = g'_m + G_{<\beta}$ for some $g'_m \in \mathcal{N} \cap G_{\le \beta}$ (for each $m=1,\dots,k$).  Set $g':= \sum_{m=1}^k r_m g'_m$, and note that $g' + G_{<\beta} = g + G_{<\beta}$.  Then:
\begin{itemize}
 \item $g' \in \langle \mathcal{N} \rangle$.  Since $g$ is also in $\langle \mathcal{N} \rangle$, it follows that $g-g' \in \langle \mathcal{N} \rangle$.
 \item $g-g' \in G_{<\beta}$.
\end{itemize}
Hence $g-g' \in \langle \mathcal{N} \rangle^F \cap G_{<\beta}$, so by the induction hypothesis, $g-g'$ is a linear combination of members of $\mathcal{N} \cap G_{<\beta}$; say
\[
g-g' = \sum_i r_i h_i
\]
where each $h_i \in \mathcal{N} \cap G_{<\beta}$ and $r_i \in R$.  Then
\[
g = g' + \sum_i r_ih_i
= \sum_m r_m g'_m + \sum_i r_ih_i \]
is a linear combination of members of $\mathcal{N} \cap G_{\le \beta}$.

\end{proof}

Finally, assume $M$ is any $R$-module.  By elementarity of $\mathcal{N}$, there is a free module $F$ and a surjective homomorphism
\[
\pi: F \to M
\]
such that $\pi \in \mathcal{N}$.  Let $\bar{G}:= \pi^{-1}[G]$, and note that $\bar{G} \in \mathcal{N}$ because both $\pi$ and $G$ are in $\mathcal{N}$.  By the result for free modules above, 
\begin{equation}\label{eq_InclusionForGbar}
\bar{G} \cap \langle \mathcal{N} \cap F \rangle^F = \langle \mathcal{N} \cap \bar{G} \rangle^F.
\end{equation}
Suppose $g$ is an element of $G \cap \langle \mathcal{N} \cap M \rangle^M$; then $g = \sum_i r_i x_i$ for some $r_i \in R$ and $x_i \in \mathcal{N}$.  By surjectivity of $\pi$ and elementarity of $\mathcal{N}$, for each $i$ there is a $\bar{x}_i \in \mathcal{N} \cap F$ such that $\pi(\bar{x}_i) = x_i$.  Let $\bar{x}:= \sum_i r_i \bar{x}_i$ in $F$, and observe that $\pi(\bar{x}) = g$; hence, $\bar{x} \in \bar{G}$.  Also, $\bar{x} \in \langle \mathcal{N} \cap F\rangle$, so by \eqref{eq_InclusionForGbar}, $\bar{x}$ is a linear combination of members of $\mathcal{N} \cap \bar{G}$; say $\bar{x} = \sum_j t_j \bar{g}_j$ where each $\bar{g}_j \in \mathcal{N} \cap \bar{G}$.  Then $\pi(\bar{g}_j) \in \mathcal{N} \cap G$ for each $j$, so $g = \pi(\bar{x}) = \sum_j t_j \pi(\bar{g}_j)$ is an element of $\langle \mathcal{N} \cap G \rangle$.

For the other direction of the theorem, suppose $R$ has a unit and $\kappa$ is (regular) and uncountable.  Let $G$ be an ideal in $R$ that is not $<\kappa$-generated; we show that $G$ and $M:=R$ provide a counterexample to part \ref{item_AllElemSubm}.  Since $\kappa$ is regular and uncountable, there is an $\mathcal{N} \prec \mathcal{H}_\theta$ (for sufficiently large $\theta$) such that $R,G \in \mathcal{N}$, $|\mathcal{N}|<\kappa$, and $\mathcal{N} \cap \kappa \in \kappa$.  By elementarity of $\mathcal{N}$ and the fact that $R \in \mathcal{N}$, it follows that $1_R \in \mathcal{N}$, and hence $\langle \mathcal{N} \cap R \rangle^R_R = R$.  So $G \cap \langle \mathcal{N} \cap R \rangle^R_R = G$.  On the other hand, $\langle \mathcal{N} \cap G \rangle^M_R$ cannot contain $G$, since $|\mathcal{N} \cap G| \le |\mathcal{N}|<\kappa$ and $G$ is not $<\kappa$-generated.
\end{proof}

If $A$ is an $R$-submodule of $B$ and $\{ R,A,B \} \subset \mathcal{N} \prec \mathcal{H}_\theta$, it is easy to show that the factor map $b \mapsto b + A$, when  restricted to the domain $\langle \mathcal{N} \cap B \rangle$, has kernel $A \cap \langle \mathcal{N} \cap B \rangle$ and maps onto $\left\langle \mathcal{N} \cap \frac{B}{A} \right\rangle$, and hence that
\[
 \frac{\langle \mathcal{N} \cap B \rangle}{A \cap \langle \mathcal{N} \cap B \rangle} \simeq \left\langle \mathcal{N} \cap \frac{B}{A} \right\rangle.
\]
If $R$ is $<\kappa$-Noetherian and $\mathcal{N} \cap \kappa$ is transitive, things work out a little more nicely:
\begin{lemma}\label{lem_IntersectQuotient}
Suppose $R$ is $<\kappa$-Noetherian, $A$ is an $R$-submodule of $B$, $\{ R,A,B \} \subset \mathcal{N} \prec \mathcal{H}_\theta$, and $\mathcal{N} \cap \kappa$ is transitive.  Then the map
\[
\sum_i r_i x_i + \langle \mathcal{N} \cap A \rangle \mapsto \sum_i r_i x_i + A
\]
(where each $r_i \in R$ and $x_i \in \mathcal{N} \cap B$) is a well-defined isomorphism from $\frac{\langle \mathcal{N} \cap B \rangle }{    \langle \mathcal{N} \cap A \rangle }$ onto $\left \langle \mathcal{N} \cap \frac{B}{A} \right \rangle$.
\end{lemma}
\begin{proof}
That the map is well-defined follows from the fact that $\langle \mathcal{N} \cap A \rangle \subseteq A$.  To see that the map is injective, suppose 
\[
\sum_i r_i x_i + A = \sum_j s_j y_j + A
\]
where the $r_i$'s and $s_j$'s come from $R$, and the $x_i$'s and $y_j$'s come from $\mathcal{N} \cap B$.  Then
\[
\sum_i r_i x_i - \sum_j s_j y_j \in A \cap \langle \mathcal{N} \cap B \rangle = \langle \mathcal{N} \cap A \rangle,
\]
where the right equality is by Theorem \ref{thm_CharactNoeth}.  The map is surjective because any coset in $\mathcal{N} \cap \frac{B}{A}$ is, by elementarity of $\mathcal{N}$, of the form $x + A$ for some $x \in \mathcal{N}$.
\end{proof}

\begin{lemma}\label{lem_ExactRestrctNoProjective} Suppose $\kappa$ is regular, $R$ is $<\kappa$-Noetherian, and $\mathcal{N} \prec \mathcal{H}_\theta$ is such that $R \in \mathcal{N}$ and $\mathcal{N} \cap \kappa$ is transitive.  Suppose $M_\bullet = \big( f_n: M_n \to M_{n+1} \big)_{n \in \mathbb{Z}}$ is an exact complex of $R$-modules, and $M_\bullet \in \mathcal{N}$.  Then:
\begin{enumerate}
 \item\label{item_ExactRestrictAndQuot}  $M_\bullet \restriction \mathcal{N}$ and $M_\bullet / \mathcal{N}$ (as defined in Lemma \ref{lem_restrictComplexN})  are both exact.

 \item\label{item_KernelRestrict} The kernel of the $n$-th map in $M_\bullet \restriction \mathcal{N}$ is $\langle \mathcal{N} \cap \text{ker}(f_n) \rangle$. 
 \item\label{item_KernelIsoQuot}    The kernel of the $n$-th map in $M_\bullet / \mathcal{N}$ is isomorphic to
   \[
\frac{  \text{ker}(f_n)}{ \big\langle \mathcal{N} \cap \text{ker}(f_n) \big\rangle};
  \]
 \item\label{item_ProjComplexRestrictProj} If $M_\bullet$ consists entirely of projective modules, then so do $M_\bullet \restriction \mathcal{N}$ and $M_\bullet / \mathcal{N}$.
 
 \item\label{item_QuotientComplex} If $Q_\bullet$ is a subcomplex of $M_\bullet$ and $Q_\bullet$ is also an element of $\mathcal{N}$, then 
\[
\frac{M_\bullet}{Q_\bullet} \restriction \mathcal{N} \ \simeq \ \frac{M_\bullet \restriction \mathcal{N}}{Q_\bullet \restriction \mathcal{N}}.
\]
\end{enumerate}
\end{lemma}
\begin{proof}
Fix any $n \in \mathbb{Z}$; note that $n \in \mathcal{N}$ by Fact \ref{fact_BasicElemSub}. Both $\text{ker}(f_n)$ and $\text{im}(f_{n-1})$ are elements of $\mathcal{N}$, because $M_\bullet \in \mathcal{N}$ and they are definable from the parameters $M_\bullet$ and $n$.  Then 
\begin{equation}\label{eq_PullKernelIn}
\text{ker} \Big( f_n \restriction \langle \mathcal{N} \cap M_n \rangle \Big) = \text{ker}(f_n) \cap \langle \mathcal{N} \cap M_n \rangle = \langle \mathcal{N} \cap \text{ker}(f_n) \rangle,
\end{equation}
where the right equality is by Theorem \ref{thm_CharactNoeth} and the left equality is obvious.  This takes care of part \eqref{item_KernelRestrict}.

We claim that
\begin{equation}\label{eq_ChainOfIms}
 \text{im}\Big( f_{n-1} \restriction \langle \mathcal{N} \cap M_{n-1} \rangle \Big) =  \text{im}(f_{n-1}) \cap \langle \mathcal{N} \cap M_n \rangle = \langle \mathcal{N} \cap \text{im}(f_{n-1}) \rangle.
\end{equation}
The right equality is immediate, by Theorem \ref{thm_CharactNoeth}.  The $\subseteq$ direction of the left equality follows from Lemma \ref{lem_restrictComplexN}.  To see the $\supseteq$ direction of the left equality, pick any $y \in \langle \mathcal{N} \cap \text{im}(f_{n-1}) \rangle$; then $y=\sum_i r_i y_i$ for some $r_i \in R$ and some $y_i \in \mathcal{N} \cap \text{im}(f_{n-1})$.  By elementarity of $\mathcal{N}$, for each $i$ there is some $x_i \in \mathcal{N}  \cap M_{n-1}$ such that $f_{n-1}(x_i) = y_i$.  Then $x:= \sum_i r_i x_i$ is an element of $\langle \mathcal{N} \cap M_{n-1} \rangle$, and $f_{n-1}(x) = y$.  So $y$ is in the image of $f_{n-1} \restriction \langle \mathcal{N} \cap M_{n-1} \rangle$.

Since $M_\bullet$ was exact by assumption, $\text{ker}(f_n) = \text{im}(f_{n-1})$, and hence the rightmost terms in \eqref{eq_PullKernelIn} and \eqref{eq_ChainOfIms} are equal to each other.  So
\[
\text{ker}\Big( f_n \restriction \langle \mathcal{N} \cap M \rangle  \Big) = \text{im} \Big(  f_{n-1} \restriction \langle \mathcal{N} \cap M_{n-1} \rangle \Big),
\]
yielding that $M_\bullet \restriction \mathcal{N}$ is exact.  So
\[
\xymatrix{
0_\bullet \ar[r] & M_\bullet \restriction \mathcal{N} \ar[r]^-{\text{id}_\bullet} & M_\bullet \ar[r]^-{\pi_\bullet} & M_\bullet / \mathcal{N} \ar[r] & 0_\bullet 
}
\]
is a short exact sequence of complexes, and both $M_\bullet\restriction \mathcal{N}$ and $M_\bullet$ are exact.  Then $M_\bullet / \mathcal{N}$ is also exact, by the 3-by-3 lemma for complexes (\cite{MR1269324}).

For part \eqref{item_KernelIsoQuot}:  abstract nonsense tells us that the kernel of the $n$-th map of $M_\bullet / \mathcal{N}$ is isomorphic to 
\begin{equation}\label{eq_FormWeDONTneed}
\frac{\text{ker}(f_n)}{\text{ker} \Big( f_n \restriction \langle \mathcal{N} \cap M_n \rangle \Big)},
\end{equation}
and by \eqref{eq_PullKernelIn} this is equal to 
  \begin{equation}\label{eq_FormWeNeed}
\frac{  \text{ker}(f_n)}{ \Big\langle \mathcal{N} \cap \text{ker}(f_n) \Big\rangle}.
  \end{equation}
Note:  we need form \eqref{eq_FormWeNeed}, not form \eqref{eq_FormWeDONTneed}, in order to apply Theorem \ref{thm_CharacterizeDecon} later on.

Part \eqref{item_ProjComplexRestrictProj} follows immediately from Lemma \ref{lem_Kaplansky}.

For part \eqref{item_QuotientComplex}, Lemma \ref{lem_IntersectQuotient} implies that for each $n \in \mathbb{Z}$, the map
\[
\sigma_n:   \frac{\langle \mathcal{N} \cap M_n \rangle}{\langle \mathcal{N} \cap Q_n \rangle} \to \left\langle \mathcal{N} \cap \frac{M_n}{Q_n} \right\rangle 
\]
defined by
\[
\sum_i r_i x_i + \langle \mathcal{N} \cap Q_n \rangle \mapsto \sum_i r_i x_i + Q_n
\]
---where the $x_i$'s are from $\mathcal{N} \cap M_n$---is an isomorphism.  If $f'_n$ denotes the $n$-th map of the complex $\frac{M_\bullet }{Q_\bullet } \restriction \mathcal{N}$, and $f''_n$ denotes the $n$-th map of the complex $\frac{M_\bullet \restriction \mathcal{N}}{Q_\bullet \restriction \mathcal{N}}$, then it is routine to check that
\[
f''_n = \sigma_{n+1}^{-1} \circ  f'_n   \circ \sigma_n.
\]
Hence, 
\[
\sigma_\bullet: \frac{M_\bullet \restriction \mathcal{N}}{Q_\bullet \restriction \mathcal{N}} \ \to \ \frac{M_\bullet}{Q_\bullet} \restriction \mathcal{N}
\]
is an isomorphism of complexes.  

\end{proof}

For a cardinal $\mu$, a module $M$ is called \textbf{strongly $\mu$-presented} if it has a projective resolution consisting of $\mu$-generated modules.  
\begin{lemma}\label{lem_StrongGen}
Suppose $\kappa$ is regular, $R$ is $<\kappa$-Noetherian, $M$ is any $R$-module, and $\mathcal{N} \prec \mathcal{N}' \prec \mathcal{H}_\theta$ are such that both $\mathcal{N}$ and $\mathcal{N}'$ have transitive intersection with $\kappa$, and $M,R \in \mathcal{N}$.  Then $\langle \mathcal{N} \cap M \rangle$ is strongly $|\mathcal{N}|$-presented and 
\[
\frac{\langle \mathcal{N}' \cap M \rangle}{\langle \mathcal{N} \cap M \rangle}
\]
is strongly $|\mathcal{N}'|$-presented.
\end{lemma}
\begin{proof}
By elementarity of $\mathcal{N}$ in $\mathcal{H}_\theta$, there is a projective resolution
\[
\xymatrix{
P_\bullet: \ \dots \ar[r] & P_2 \ar[r]^-{f_2} & P_1 \ar[r]^-{f_1} & P_0 \ar[r]^-{f_0} & M \ar[r] & 0
}
\]
of $M$ such that $P_\bullet \in \mathcal{N}$.  By Lemma \ref{lem_ExactRestrctNoProjective}, 
$P_\bullet \restriction \mathcal{N}$ and $P_\bullet \restriction \mathcal{N}'$ are exact.  Then the 3-by-3 lemma for complexes applied to the exact sequence
\[
\xymatrix{
0_\bullet \ar[r] & P_\bullet \restriction \mathcal{N} \ar[r]^-{\text{id}_\bullet} & P_\bullet \restriction \mathcal{N}' \ar[r] & (P_\bullet \restriction \mathcal{N}')/(P_\bullet \restriction \mathcal{N}) \ar[r] & 0_\bullet
}
\]
yields that $(P_\bullet \restriction \mathcal{N}')/(P_\bullet \restriction \mathcal{N})$ is also exact.  Lemma \ref{lem_Kaplansky} ensures that for each $n \ge 0$, $\langle \mathcal{N} \cap P_n \rangle$ and $\frac{\langle \mathcal{N}' \cap P_n \rangle}{\langle \mathcal{N} \cap P_n \rangle}$ are projective.  So $P_\bullet \restriction \mathcal{N}$ is a projective resolution of $\langle \mathcal{N} \cap M \rangle$, and $(P_\bullet \restriction \mathcal{N}')/(P_\bullet \restriction \mathcal{N})$ is a projective resolution of $\langle \mathcal{N}' \cap M \rangle / \langle \mathcal{N} \cap M \rangle$.  Terms in $P_\bullet \restriction \mathcal{N}$ are obviously $|\mathcal{N}|$-generated, and terms in $(P_\bullet \restriction \mathcal{N}')/(P_\bullet \restriction \mathcal{N})$ are obviously $|\mathcal{N}'|$-generated.
 
\end{proof}

\section{Proof of Theorem \ref{thm_CharacterizeDecon} (characterization of deconstructibility)}\label{sec_ProofCharDecon}

Given a collection $\mathcal{C}$ of modules over a fixed ring $R$, a \textbf{$\mathcal{C}$-filtration} is a $\subseteq$-increasing and $\subseteq$-continuous sequence $\langle M_\xi \ : \ \xi < \eta \rangle$ of $R$-modules such that $M_0 = 0$ and for all $\xi < \eta$ such that $\xi +1 < \eta$:  $M_\xi$ is a submodule of $M_{\xi+1}$ and $M_{\xi+1}/M_\xi$ is isomorphic to an element of $\mathcal{C}$.\footnote{Of course if $\mathcal{C}$ is closed under isomorphism we could just say ``in" $\mathcal{C}$, but it will be often convenient to view $\mathcal{C}$ as a set, rather than a proper class.}  A module $M$ is called \textbf{$\boldsymbol{\mathcal{C}}$-filtered} if it has a $\mathcal{C}$-filtration; i.e.\ if there exists a $\mathcal{C}$-filtration $\vec{M}$ whose union is $M$.  The class of $\mathcal{C}$-filtered modules is denoted $\text{Filt}(\mathcal{C})$.

Given a class $\mathcal{F}$ of modules and a regular cardinal $\kappa$, let $\mathcal{F}^{<\kappa}$ denote the collection of $<\kappa$-presented members of $\mathcal{F}$.  We say that \textbf{$\boldsymbol{\mathcal{F}}$ is $\boldsymbol{<\kappa}$-deconstructible} (in the sense of G\"obel-Trlifaj~\cite{MR2985554}) if every member of $\mathcal{F}$ is $\mathcal{F}^{<\kappa}$-filtered.  ``$\mathcal{F}$ is deconstructible" means there exists a $\kappa$ such that $\mathcal{F}$ is $<\kappa$-deconstructible. $\mathcal{F}$ is \textbf{closed under transfinite extensions} if $\text{Filt}(\mathcal{F}) \subseteq \mathcal{F}$.  The key fact about deconstructibility is:
\begin{theorem}[Saor\'{\i}n-\v{S}\v tov\'{\i}\v{c}ek~\cite{MR2822215}]\label{thm_StovDecon}
If $\mathcal{F}$ is deconstructible and closed under transfinite extensions, then it is a precovering class.
\end{theorem}

We introduce a couple of ad-hoc definitions:  $\mathcal{F}$ is \textbf{strongly $\boldsymbol{<\kappa}$-deconstructible} if every member of $\mathcal{F}$ is $\mathcal{F}^{<\kappa}_s$-filtered, where $\mathcal{F}^{<\kappa}_s$ denotes the collection of strongly $<\kappa$-presented members of $\mathcal{F}$---those members of $\mathcal{F}$ that have a projective resolution consisting entirely of $<\kappa$-generated modules.  Similarly, let $\mathcal{F}^{<\kappa}_{\text{g}}$ denote the collection of $<\kappa$-generated (but not necessarily $<\kappa$-presented) members of $\mathcal{F}$, and let us say that $\mathcal{F}$ is \textbf{weakly $\boldsymbol{<\kappa}$-deconstructible} if every member of $\mathcal{F}$ is $\mathcal{F}^{<\kappa}_{\text{g}}$-filtered.  Note that if $\mathcal{F}$ is weakly $<\kappa$-deconstructible, then by picking $\lambda$ large enough that every $F \in \mathcal{F}_{\text{g}}^{<\kappa}$ is strongly $<\lambda$-presented,  it follows that $\mathcal{F}$ is strongly $<\lambda$-deconstructible.  So ``$\mathcal{F}$ is deconstructible" could have equivalently been defined as the existence of some $\kappa$ such that $\mathcal{F}$ is weakly $<\kappa$-deconstructible, or the existence of some $\kappa$ such that $\mathcal{F}$ is strongly $<\kappa$-deconstructible.  The distinction between the various notions is only relevant if one is interested in what happens at a particular $\kappa$ (typically $\kappa = \aleph_1$).

\begin{remark}\label{rem_AlternativeCharDecon}
Let \ref{item_KappaDecon}$_{,w}$ be the result of replacing ``$<\kappa$-deconstructible" with ``weakly $<\kappa$-deconstructible" in the statement of part \ref{item_KappaDecon} of Theorem \ref{thm_CharacterizeDecon}.  Then, even if we omit the $<\kappa$-Noetherian assumption on $R$ in the background assumptions of Theorem \ref{thm_CharacterizeDecon}, a minor variant of the proof below still shows that \ref{item_ElemSub} implies \ref{item_KappaDecon}$_{,w}$; see footnote \ref{fn_WeakVersion} on page \pageref{fn_WeakVersion}.  Then the implication \ref{item_ElemSub} $\implies$ \ref{item_KappaDecon}$_{,w}$, combined with Lemma \ref{lem_Kaplansky}, yields an alternative proof of the classic Kaplansky's Theorem (using $\kappa = \aleph_1$).
\end{remark}

\begin{remark}\label{rem_TFextNotNeeded}
For the \ref{item_KappaDecon} $\implies$ \ref{item_ElemSub} direction of Theorem \ref{thm_CharacterizeDecon} (under the assumption that $\mathcal{F}$ is closed under transfinite extensions), it suffices to assume $\mathcal{F}$ is \emph{weakly} $<\kappa$-deconstructible. 
\end{remark}

The \ref{item_ElemSub_PAIRS} $\implies$ \ref{item_ElemSub} direction of Theorem \ref{thm_CharacterizeDecon} is trivial, by considering $\mathcal{N}':= H_\theta$ (in which case $\langle \mathcal{N}' \cap M \rangle$ is just $M$).  So we only need to prove:
\begin{itemize}
 \item  \ref{item_ElemSub} $\implies$ \ref{item_KappaDecon}; and
 \item \ref{item_KappaDecon} $\implies$ \ref{item_ElemSub_PAIRS} (assuming that $\mathcal{F}$ is also closed under transfinite extensions).
\end{itemize}

\subsection{Proof of \ref{item_ElemSub} $\implies$ \ref{item_KappaDecon} direction of Theorem \ref{thm_CharacterizeDecon}}

Assume that \ref{item_ElemSub} holds in the statement of Theorem \ref{thm_CharacterizeDecon}.  We prove that \ref{item_KappaDecon} holds by induction; more precisely, we prove by induction on cardinals $\lambda$ that if $M \in \mathcal{F}$ is strongly $ \lambda$-presented, then $M$ is strongly $\mathcal{F}^{<\kappa}$-filtered.

For $\lambda < \kappa$ the desired statement trivially holds; i.e.\ if $M \in \mathcal{F}$ and is strongly $\lambda$-presented for some $\lambda < \kappa$, then $M_0  =0$ and $M_1 = M$ is the desired $\mathcal{F}^{<\kappa}_s$ filtration of $M$.

Now suppose $\lambda \ge \kappa$ (possibly $\lambda$ is singular) and that:
\begin{quote}
$\textbf{IH}_{\boldsymbol{<\lambda}}$: For all $\mu < \lambda$, all strongly $\mu$-presented members of $\mathcal{F}$ are $\mathcal{F}^{<\kappa}_s$-filtered.  
\end{quote}

Suppose $M \in \mathcal{F}$ and that $M$ is strongly $\lambda$-presented; in particular there is an $X \subseteq M$ of size at most $\lambda$ such that $M = \langle X \rangle$.  We can without loss of generality assume $X \subseteq \lambda$.   Fix any regular $\theta$ such that $R$, $M$, and $\lambda$ are elements of $H_\theta$.  Set
\[
\mathfrak{B}:=(H_\theta,\in,\kappa,R,M,X,\lambda, \mathcal{F} \cap H_\theta).
\]
By Fact \ref{fact_SmoothChain}, there exists a smooth $\in$-chain $\vec{\mathcal{N}}=\left \langle \mathcal{N}_\xi \ : \ \xi <  \text{cf}(\lambda)  \right\rangle$ of elementary submodels of $\mathfrak{B}$ such that each $\mathcal{N}_\xi$ has cardinality $<\lambda$, has transitive intersection with $\kappa$, and such that 
\[
X \subseteq \lambda \subseteq \bigcup_{\xi < \text{cf}(\lambda)} \mathcal{N}_\xi.
\]

In order to simplify notation when using $\vec{\mathcal{N}}$ to create an $\mathcal{F}$-filtration of $M$, we will insist that $\mathcal{N}_0$ is not an elementary submodel of $\mathfrak{B}$, but merely $\mathcal{N}_0 = \{ 0_M \}$.  This makes no substantial difference, but simplifies notation later on.  Note that since $\mathcal{N}_\xi$, $M$, and $R$ are elements of $\mathcal{N}_{\xi+1}$, both $\langle \mathcal{N}_\xi \cap M \rangle$ and $M/\langle N_\xi \cap M \rangle$ are elements of $\mathcal{N}_{\xi+1}$.  Then, by Lemma \ref{lem_IntersectQuotient},

\begin{equation}\label{eq_SmoothChainQuot}
\left\langle \mathcal{N}_{\xi+1} \cap \frac{M}{\langle \mathcal{N}_\xi \cap M \rangle} \right\rangle \simeq \frac{\langle \mathcal{N}_{\xi+1} \cap M \rangle }{\langle \mathcal{N}_{\xi+1} \cap M \rangle \cap \langle \mathcal{N}_\xi \cap M \rangle. }
\end{equation}
But since $N_\xi$ is also a subset of $N_{\xi+1}$, the denominator of the right side of \eqref{eq_SmoothChainQuot} is equal to $\langle \mathcal{N}_\xi \cap M \rangle$.  Hence,
\begin{equation}\label{eq_SmoothChainQuot_2}
\left\langle \mathcal{N}_{\xi+1} \cap \frac{M}{\langle \mathcal{N}_\xi \cap M \rangle} \right\rangle \simeq \frac{\langle \mathcal{N}_{\xi+1} \cap M \rangle }{ \langle \mathcal{N}_\xi \cap M \rangle.
}
\end{equation}

Now $M \in \mathcal{F}$ and for each $\xi$, $M \in \mathcal{N}_\xi \prec \mathfrak{B}$; so $\frac{M}{\langle \mathcal{N}_\xi \cap M \rangle}$ is an element of $\mathcal{F}$ by assumption \ref{item_ElemSub}.  It is also an element of $\mathcal{N}_{\xi+1}$, as noted above.  Then again by assumption \ref{item_ElemSub}---this time applied to $\mathcal{N}_{\xi+1}$ and its element $\frac{M}{\langle \mathcal{N}_\xi \cap M \rangle} \in \mathcal{F}$---it follows that the left side of \eqref{eq_SmoothChainQuot_2} is an element of $\mathcal{F}$.  Hence,  
\begin{equation}\label{eq_NotQuiteOurFilt}
\Big \langle \langle \mathcal{N}_\xi \cap M  \rangle \ : \ \xi < \text{cf}(\lambda)  \Big\rangle
\end{equation}
is an $\mathcal{F}$-filtration of $M$,\footnote{Note that by our insistence that $\mathcal{N}_0 := \{ 0 \}$, the 0-th entry of \eqref{eq_NotQuiteOurFilt} is just $\{ 0 \}$ and entry 1 is $\langle \mathcal{N}_1 \cap M \rangle$, so their quotient is isomorphic to $\langle \mathcal{N}_1 \cap M \rangle$ which is in $\mathcal{F}$ by assumption. } though some of the adjacent factors may be too large; i.e.\ it may fail to be a $\mathcal{F}^{<\kappa}_s$-filtration.

Consider any fixed $\xi < \text{cf}(\lambda)$.  By Lemma \ref{lem_StrongGen}, our assumption that $R$ is $<\kappa$-Noetherian, and the fact that $\mathcal{N}_\xi$ and $\mathcal{N}_{\xi+1}$ both have transitive intersection with $\kappa$, the quotient in \eqref{eq_SmoothChainQuot_2} is strongly $|\mathcal{N}_{\xi+1}|$-presented.\footnote{Even if we had omitted the assumption that $R$ is $<\kappa$-Noetherian, the quotient on the right side of \eqref{eq_SmoothChainQuot_2} is obviously $ |\mathcal{N}_{\xi+1}|$-generated (though possibly not $ |\mathcal{N}_{\xi+1}|$-presented).  This, together with the obvious adjustment to the induction hypothesis, is the only difference between the current proof and the proof of the implication \ref{item_ElemSub} $\implies$  \ref{item_KappaDecon}$_{,w}$ mentioned in Remark \ref{rem_AlternativeCharDecon}. \label{fn_WeakVersion} }  Since the quotient in \eqref{eq_SmoothChainQuot_2} also belongs to $\mathcal{F}$ and $|\mathcal{N}_{\xi+1}|<\lambda$, our induction hypothesis $\text{IH}_{<\lambda}$ applies to \eqref{eq_SmoothChainQuot_2}, yielding some filtration
\[
\left\langle Z^\xi_\zeta / \langle N_\xi \cap M \rangle \ : \ \zeta < \eta^\xi  \right\rangle
\]
of the right side of \eqref{eq_SmoothChainQuot_2} such that 
\[
\frac{Z^\xi_{\zeta+1} / \langle N_\xi \cap M \rangle}{Z^\xi_\zeta / \langle N_\xi \cap M \rangle} \ \simeq \ \frac{Z^\xi_{\zeta+1}}{Z^\xi_\zeta}
\]
is in $\mathcal{F}^{<\kappa}_s$ for each $\zeta$ such that $\zeta + 1 < \eta^\xi$.  For each $\xi < \text{cf}(\lambda)$ let
\[
\vec{Z}^\xi:= \langle Z^\xi_\zeta \ : \ \zeta < \eta^\xi \rangle.
\]
Then concatenating the $\vec{Z}^\xi$'s across all $\xi < \text{cf}(\lambda)$ yields the desired $\mathcal{F}^{<\kappa}_s$-filtration of $M=\bigcup_{\xi < \text{cf}(\lambda)} \langle \mathcal{N}_\xi \cap M \rangle$.
 
This completes the proof of the \ref{item_ElemSub} $\implies$ \ref{item_KappaDecon} direction of Theorem \ref{thm_CharacterizeDecon}.

\subsection{Proof of the \ref{item_KappaDecon} $\implies$ \ref{item_ElemSub_PAIRS} direction of Theorem \ref{thm_CharacterizeDecon}}

Suppose $\kappa$ is regular and uncountable, and that $\mathcal{F}$ is a weakly $<\kappa$-deconstructible class that is closed under transfinite extensions.  Let $\mathcal{F}_\theta:= \mathcal{F} \cap H_\theta$ (viewed as a predicate on $H_\theta$) and consider any
\begin{equation}\label{eq_N_elem_again}
\mathcal{N} \prec \mathcal{N}' \prec (H_\theta,\in,\kappa,R,\mathcal{F}_\theta)
\end{equation}
such that $\mathcal{N} \cap \kappa$ and $\mathcal{N}' \cap \kappa$ are both transitive.  Suppose $M \in \mathcal{F} \cap \mathcal{N}$.  By the assumption that $\mathcal{F}$ is weakly $<\kappa$-deconstructible, there exists a $\mathcal{F}^{<\kappa}_{\text{g}}$-filtration of $M$.  This can be expressed in the structure $(H_\theta,\in,\kappa,R,\mathcal{F}_\theta)$, and so by \eqref{eq_N_elem_again}, there is a $\mathcal{F}^{<\kappa}_{\text{g}}$-filtration 
\begin{equation}
\vec{M} = \langle M_\xi \ : \ \xi < \eta \rangle 
\end{equation}
of $M$ that is an \emph{element} of $\mathcal{N}$.

Elementarity of $\mathcal{N}$, together with the fact that $\vec{M}$ is a filtration of $M$, ensures that 
\begin{equation}
\Big\langle \langle \mathcal{N} \cap M_\xi \rangle \ : \ \xi \in \mathcal{N} \cap \eta \Big\rangle
\end{equation}
is a filtration of $\langle \mathcal{N} \cap M \rangle$.\footnote{To see continuity, note that if $0 \ne \xi \in \mathcal{N} \cap \eta$ and there is no largest element of $\mathcal{N} \cap \xi$, then by elementarity of $\mathcal{N}$, $\xi$ must be a limit ordinal.  Hence, by continuity of $\vec{M}$, $M_\xi = \bigcup_{\zeta < \xi} M_\zeta$.  It follows by elementarity of $\mathcal{N}$ that $\mathcal{N} \cap M_\xi =  \bigcup_{\zeta \in \mathcal{N} \cap \xi} (\mathcal{N} \cap M_\zeta)$.  Finally, since $\mathcal{N} \cap \xi$ has no largest element, this last equality implies $\langle \mathcal{N} \cap M_\xi \rangle =  \bigcup_{\zeta \in \mathcal{N} \cap \xi} \langle \mathcal{N} \cap M_\zeta \rangle$. }  We claim that adjacent quotients from it are in $\mathcal{F}$; it will then follow from closure of $\mathcal{F}$ under transfinite extensions that $\langle \mathcal{N} \cap M \rangle$ is in $\mathcal{F}$.  Indeed, suppose $\xi \in \mathcal{N} \cap \eta$ and $\xi+1 < \eta$; then $\xi+1 \in \mathcal{N}$ too, so both $M_\xi$ and $M_{\xi+1}$ are elements of $\mathcal{N}$.  Since $\mathcal{N} \cap \kappa$ is transitive and $M_{\xi+1} / M_\xi$ is $<\kappa$-generated and an element of $\mathcal{N}$, Fact \ref{fact_BasicElemSub} implies that there is a generating set for $M_{\xi+1}/M_\xi$ that is both an element and a subset of $\mathcal{N}$; in particular,
\[
\left\langle \mathcal{N} \cap \frac{M_{\xi+1}}{M_\xi} \right\rangle = \frac{M_{\xi+1}}{M_\xi}.
\]
And, since $R$ is $<\kappa$-Noetherian and $\mathcal{N} \cap \kappa$ is transitive, Lemma \ref{lem_IntersectQuotient} ensures
\[
\left\langle \mathcal{N} \cap \frac{M_{\xi+1}}{M_\xi} \right\rangle \ \simeq \ \frac{\langle \mathcal{N} \cap M_{\xi+1} \rangle}{\langle \mathcal{N} \cap M_{\xi} \rangle}
\]
So
\[
\frac{\langle \mathcal{N} \cap M_{\xi+1} \rangle}{\langle \mathcal{N} \cap M_{\xi} \rangle} \ \simeq \ \frac{M_{\xi+1}}{M_\xi}
\]
which is a member of $\mathcal{F}$.  This completes the proof that $\langle \mathcal{N} \cap M \rangle$ is $\mathcal{F}$-filtered (in fact, $\mathcal{F}^{<\kappa}_{\text{g}}$-filtered) and hence, by the assumed closure of $\mathcal{F}$ under transfinite extensions, an element of $\mathcal{F}$.

Next we will show that $\langle \mathcal{N}' \cap M \rangle/\langle \mathcal{N} \cap M \rangle$ is $\mathcal{F}$-filtered (in fact, $\mathcal{F}^{<\kappa}_{\text{g}}$-filtered); since $\mathcal{F}$ is assumed to be closed under transfinite extensions, this will complete the proof.  Since $\vec{M} \in \mathcal{N}'$, the sequence

\begin{equation}\label{eq_FiltOfQuotient100}
\left\langle \frac{\big\langle (\mathcal{N} \cap M) \cup (\mathcal{N}' \cap M_\xi) \big\rangle }{\langle \mathcal{N} \cap M \rangle} \ : \  \xi \in \mathcal{N}' \cap \eta \right\rangle
\end{equation}
is a filtration of $\langle \mathcal{N}' \cap M \rangle/\langle \mathcal{N} \cap M \rangle$.  We will show that its adjacent quotients are in $\mathcal{F}^{<\kappa}_{\text{g}}$; i.e.:
\begin{nonGlobalClaim}\label{clm_QuotOfUnions}
If $\xi \in \mathcal{N}' \cap \eta$, then
\begin{equation}\label{eq_QuotientToCheck}
\frac{\big\langle (\mathcal{N} \cap M ) \cup (\mathcal{N}' \cap M_{\xi+1}) \big\rangle}{\big\langle (\mathcal{N} \cap M ) \cup (\mathcal{N}' \cap M_{\xi}) \big\rangle}
\end{equation}
is an element of $\mathcal{F}^{<\kappa}_{\text{g}}$.  
\end{nonGlobalClaim}
\begin{proof}
(of Claim \ref{clm_QuotOfUnions}):  Note that since $\xi \in \mathcal{N}'$, the ordinal $\xi+1$ is also in $\mathcal{N}'$.  We consider cases.

\textbf{Case 1: $\boldsymbol{\xi \in \mathcal{N}}$.}  Then by elementarity of $\mathcal{N}$, $\xi+1$ is also in $\mathcal{N}$, and hence both $M_\xi$ and $M_{\xi+1}$ are elements of $\mathcal{N}$.  Since $M_{\xi+1}/M_\xi$ is $<\kappa$-generated, there is an $X_\xi \subset M_{\xi+1}$ of size $<\kappa$ such that $M_{\xi+1} = M_\xi + X_\xi$.  And by elementarity and our case, we can take this $X_\xi$ to be an element of $\mathcal{N}$.  Since $|X_\xi|<\kappa$ and $\mathcal{N}$ has transitive intersection with $\kappa$, it follows from Fact \ref{fact_BasicElemSub} that 
 \begin{equation}\label{eq_X_xi_in_N}
 X_\xi \subset \mathcal{N} \ (\cap M_{\xi+1}).
 \end{equation}
Note that $X_\xi$ is in $\mathcal{N}'$ too, because $\mathcal{N} \subset \mathcal{N}'$.  Now if $z$ is an element of $\mathcal{N}' \cap M_{\xi+1} = \mathcal{N}' \cap  (M_\xi + X_\xi)$, then by elementarity of $\mathcal{N}'$, $z = m + x$ for some $m \in \mathcal{N}' \cap M_\xi$ and some $x \in \mathcal{N}' \cap X_\xi$.  By \eqref{eq_X_xi_in_N}, $x$ is an element of $\mathcal{N}$.  Hence $z$ is an element of $\langle (\mathcal{N} \cap M) \cup (\mathcal{N}' \cap M_\xi) \rangle$.  So
 \[
\Big\langle  (\mathcal{N} \cap M) \cup (\mathcal{N}' \cap M_{\xi+1}) \Big\rangle = \Big\langle (\mathcal{N} \cap M) \cup \big( \mathcal{N}' \cap (M_\xi + X_\xi) \big) \Big\rangle = \Big\langle (\mathcal{N} \cap M) \cup (\mathcal{N}' \cap  M_\xi) \Big\rangle,
 \]
and hence the numerator is identical to the denominator in \eqref{eq_QuotientToCheck}.  And we can without loss of generality assume the trivial module is in $\mathcal{F}$.

\textbf{Case 2: $\boldsymbol{\xi \notin \mathcal{N}}$.}  We prove in this case that \eqref{eq_QuotientToCheck} is isomorphic to $M_{\xi+1}/M_\xi$, which is in $\mathcal{F}^{<\kappa}_{\text{g}}$ by assumption.    Define
\[
\Phi: \langle \mathcal{N}' \cap M_{\xi+1} \rangle \to \frac{\big\langle (\mathcal{N} \cap M ) \cup (\mathcal{N}' \cap  M_{\xi+1}) \big\rangle}{\big\langle (\mathcal{N} \cap M ) \cup (\mathcal{N}' \cap  M_{\xi}) \big\rangle} = \frac{\Big\langle \big( \mathcal{N} \setminus M_{\xi+1} \big) \cup (\mathcal{N}' \cap M_{\xi+1}) \Big\rangle}{\Big\langle \big( \mathcal{N} \setminus M_{\xi} \big) \cup (\mathcal{N}' \cap  M_{\xi}) \Big\rangle}
\] 
by mapping $x$ to its coset.  This is clearly a homomorphism, and is surjective because every coset in the first expression of the codomain of $\Phi$ has a representative in $\langle \mathcal{N}' \cap M_{\xi+1} \rangle$.

We will prove that $\text{ker} \ \Phi = \langle \mathcal{N}' \cap M_\xi \rangle$.  It is clear that $\text{ker} \ \Phi \supseteq \langle \mathcal{N}' \cap M_\xi \rangle$.  To prove the other inclusion, let $\rho$ be the least ordinal $\ge \xi$ that is an element of $\mathcal{N}$.  Since $\rho$ and $\vec{M}$ are both elements of $\mathcal{N}$, the elementarity of $\mathcal{N}$ ensures that $M_\rho \in \mathcal{N}$.\footnote{Note $\rho \le \eta$ because $\eta = \text{lh}(\vec{M}) \in \mathcal{N}$; if $\rho = \eta$ we set $M_\rho:= M$.}  Our case ensures that 
\begin{equation}\label{eq_XiPlus1LessRho}
\xi+1 \le \rho,
\end{equation} 
and in fact (by elementarity of $\mathcal{N}$) $\rho$ must be a limit ordinal.  So $M_\rho = \bigcup_{\zeta < \rho} M_\zeta$, and the elementarity of $\mathcal{N}$ ensures that any member of $\mathcal{N} \cap M_\rho$ is an element of $M_\zeta$ for some $\zeta \in \mathcal{N} \cap \xi$.  Hence,   
\begin{equation}
\mathcal{N} \cap M_\rho = \mathcal{N} \cap M_\xi.
\end{equation}

Since $R$ is $<\kappa$-Noetherian and $\mathcal{N} \cap \kappa$ is transitive, Theorem \ref{thm_CharactNoeth} implies
\begin{equation}
\langle \mathcal{N} \cap M \rangle \cap M_\rho = \langle \mathcal{N} \cap M_\rho \rangle.
\end{equation} 

Putting this all together, we have
\begin{equation}\label{eq_MainInclusion}
\langle \mathcal{N} \cap M \rangle \cap M_\rho \subseteq \langle \mathcal{N} \cap M_\xi \rangle \subseteq M_\xi.
\end{equation}

We are now ready to prove $\text{ker}( \Phi) \subseteq \langle \mathcal{N}' \cap M_\xi \rangle$.  Suppose $x \in \text{ker} ( \Phi)$.  Then---keeping in mind that the domain of $\Phi$ is $\langle \mathcal{N}' \cap M_{\xi+1}\rangle$---we have
\[
x \in \langle \mathcal{N}' \cap M_{\xi+1} \rangle \cap \  \Big( \langle \mathcal{N} \setminus M_\xi \rangle + \langle \mathcal{N}' \cap M_\xi \rangle \Big),
\]
So, in particular, $x = n + m$ for some $n \in \langle \mathcal{N} \setminus M_\xi \rangle$ and some $m \in \langle \mathcal{N}' \cap M_\xi \rangle$.  Now $x$ and $m$ are both in $\langle \mathcal{N}' \cap M_{\xi+1} \rangle$, so $n = x-m$ is also an element of $\langle \mathcal{N}' \cap M_{\xi+1} \rangle$.  Since $n$ is also in $\langle \mathcal{N} \rangle$, we have
\[
n \in \langle \mathcal{N} \rangle \cap M_{\xi+1}  \subseteq \langle \mathcal{N} \rangle \cap M_\rho   \subseteq M_\xi,
\]
 where the first inclusion is by \eqref{eq_XiPlus1LessRho} and the second inclusion is by \eqref{eq_MainInclusion}.  So $n \in M_\xi$.  Then $x = n+m$ is in $M_\xi$, and so
 \[
 x \in \langle \mathcal{N}' \cap M_{\xi+1} \rangle     \cap M_\xi = \Big( \langle \mathcal{N}'\rangle \cap M_{\xi+1}  \Big) \cap M_\xi = \langle \mathcal{N}' \rangle  \cap M_\xi = \langle \mathcal{N}' \cap M_\xi \rangle,
 \]
 where the first and last equalities are by Lemma \ref{lem_IntersectQuotient} (and the standing assumption of the claim that $\xi \in \mathcal{N}'$).  This completes the proof that $\text{ker}(\Phi) \subseteq \langle \mathcal{N}' \cap M_\xi\rangle$.
 
We have shown that $\Phi$ is surjective and its kernel is $\langle \mathcal{N}' \cap M_\xi \rangle$.  So the codomain of $\Phi$ is isomorphic to $
 \frac{\langle \mathcal{N}' \cap M_{\xi+1} \rangle}{\langle \mathcal{N}' \cap M_{\xi} \rangle}$.  By Lemma \ref{lem_IntersectQuotient}, this quotient is isomorphic to 
 \[
 \left\langle \mathcal{N}' \cap \frac{M_{\xi+1}}{M_\xi} \right\rangle,
\]
and this is just equal to $M_{\xi+1}/M_\xi$, because $M_{\xi+1}/M_\xi$ is a $<\kappa$-generated element of $\mathcal{N}'$ and $\mathcal{N}' \cap \kappa$ is transitive.
 
\end{proof}

So by the claim, \eqref{eq_FiltOfQuotient100} is an $\mathcal{F}$-filtration of the quotient $\frac{\langle \mathcal{N}' \cap M \rangle}{\langle \mathcal{N} \cap M \rangle}$.  By closure of $\mathcal{F}$ under transfinite extensions, it follows that $\frac{\langle \mathcal{N}' \cap M \rangle}{\langle \mathcal{N} \cap M \rangle}$ is in $\mathcal{F}$.

\subsection{Characterization of deconstructibility for classes of complexes}

We state a version of Theorem \ref{thm_CharacterizeDecon} for complexes.  Recall that if $M_\bullet$ is a complex and $M_\bullet \in \mathcal{N} \prec \mathcal{H}_\theta$, the complexes $M_\bullet \restriction \mathcal{N}$ and $M_\bullet / \mathcal{N}$ were defined in Lemma \ref{lem_restrictComplexN}.
\begin{theorem}[characterization of deconstructibility for classes of complexes]\label{thm_CharacterizeDecon_COMPLEXES}
Suppose $\kappa$ is a regular uncountable cardinal, $R$ is a $<\kappa$-Noetherian ring, and $\mathcal{K}$ is a class of complexes of $R$-modules.  Consider the following statements:

\begin{enumerate}[label=(\Roman*)$_\kappa^\bullet$]
 \item\label{item_KappaDecon_COMPLEX} $\mathcal{K}$ is strongly $<\kappa$-deconstructible.
 \item\label{item_ElemSub_COMPLEX} Whenever $\mathcal{N}$ is an elementary submodel of $(H_\theta,\in,R,\kappa, \mathcal{K} \cap H_\theta )$ and $\mathcal{N} \cap \kappa $ is transitive, then for all complexes $M_\bullet$: 
 \begin{equation*}
 M_\bullet \in \mathcal{N} \cap \mathcal{K} \ \implies M_\bullet \restriction \mathcal{N}  \in \mathcal{K} \text{ and } M_\bullet / \mathcal{N} \in \mathcal{K}. 
 \end{equation*}

 \item\label{item_ElemSub_PAIRS_COMPLEX} 
 
Whenever $\mathcal{N} \subseteq \mathcal{N}'$ are both elementary submodels of $(H_\theta,\in,R,\kappa, \mathcal{K} \cap H_\theta )$ and both $\mathcal{N} \cap \kappa $ and $\mathcal{N}' \cap \kappa$ are transitive, then for all complexes $M_\bullet$: 
 \begin{equation*}
 M_\bullet \in \mathcal{N} \cap \mathcal{K} \ \implies M_\bullet \restriction \mathcal{N}  \in \mathcal{K} \text{ and } \frac{M_\bullet \restriction \mathcal{N}'}{M_\bullet \restriction \mathcal{N}} \in \mathcal{K}. 
 \end{equation*}

 \end{enumerate}

\noindent The following implications always hold:
\begin{center}
\ref{item_KappaDecon_COMPLEX}  $\Longleftarrow$ \ref{item_ElemSub_COMPLEX} $\Longleftarrow$ \ref{item_ElemSub_PAIRS_COMPLEX}.
\end{center}
If $\mathcal{K}$ is closed under transfinite extensions, then 
\begin{center}
\ref{item_KappaDecon_COMPLEX} $\iff$ \ref{item_ElemSub_COMPLEX}  $\iff$ \ref{item_ElemSub_PAIRS_COMPLEX}.
\end{center}
\end{theorem}

The proof is omitted because it is almost identical to the proof above.  For example, the analogue of the module isomorphism in  \eqref{eq_SmoothChainQuot} is the isomorphism of complexes
\begin{equation*}
\left( \frac{M_\bullet}{M_\bullet \restriction \mathcal{N}_\xi } \right) \restriction \mathcal{N}_{\xi+1}  \simeq \frac{M_\bullet \restriction \mathcal{N}_{\xi+1}}{\big(M_\bullet \restriction \mathcal{N}_\xi \big) \restriction \mathcal{N}_{\xi+1}} \ \left( =  \frac{M_\bullet \restriction \mathcal{N}_{\xi+1}}{M_\bullet \restriction \mathcal{N}_\xi  }  \right)
\end{equation*}
given by part \ref{item_QuotientComplex} of Lemma \ref{lem_ExactRestrctNoProjective}.  

\begin{remark}
We chose to work in the familiar setting of unital rings, but the proof of Theorem \ref{thm_CharacterizeDecon} makes no essential use of the assumption that the ring is unital (recall that the forward direction of Theorem \ref{thm_CharactNoeth} worked in the non-unital setting).  As pointed out by the anonymous referee, the generalization of Theorem \ref{thm_CharacterizeDecon} to the non-unital setting directly implies Theorem \ref{thm_CharacterizeDecon_COMPLEXES} as an immediate corollary, since categories of complexes can be viewed as categories of unitary modules over rings with enough idempotents.

\end{remark}

Theorem \ref{thm_CharacterizeDecon_COMPLEXES} should generalize to other categories that have some appropriate analogues of the operations $M_\bullet \restriction \mathcal{N}$ and $M_\bullet / \mathcal{N}$ (for elementary submodels $\mathcal{N}$ that have the object $M_\bullet$ as an element).

\section{Eventual almost everywhere closure under quotients and transfinite extensions}\label{sec_AEclosureQuot}

If $R$ is a $<\kappa$-Noetherian ring and $\mathcal{F}$ is a strongly $<\kappa$-deconstructible class of $R$-modules closed under transfinite extensions, then Theorem \ref{thm_CharacterizeDecon} tells us that if 
\[
\mathcal{N} \prec \mathcal{N}' \prec (H_\theta,\in,R,\kappa, \mathcal{F} \cap H_\theta) 
\]
and both $\mathcal{N}$ and $\mathcal{N}'$ have transitive intersection with $\kappa$, then for all $R$-modules $M \in \mathcal{N}$,
\begin{equation}\label{eq_MotivatingImplic}
M \in \mathcal{F} \ \implies \ \Big( \langle \mathcal{N} \cap M \rangle \in \mathcal{F}, \ \langle \mathcal{N}' \cap M \rangle \in \mathcal{F}, \text{ and } \frac{\langle \mathcal{N}' \cap M \rangle}{\langle \mathcal{N} \cap M \rangle} \in \mathcal{F} \Big) 
\end{equation}
Theorem \ref{thm_CharacterizeDecon_COMPLEXES} gives a similar consequence for strongly $<\kappa$-deconstructible classes of complexes.

We next define the notion of a class $\mathcal{F}$ being ``$\kappa$-almost everywhere closed under quotients"; this is basically what one gets by moving the statements $\langle \mathcal{N} \cap M \rangle \in \mathcal{F}$ and $\langle \mathcal{N}' \cap M \rangle \in \mathcal{F}$ from the consequent of \eqref{eq_MotivatingImplic} to the antecedent of \eqref{eq_MotivatingImplic}.

\begin{definition}\label{def_KappaAE}
Let $\mathcal{K}$ be a class of complexes of $R$-modules, and $\kappa$ a regular uncountable cardinal.  We will say that:
\begin{enumerate}[label=(\alph*)]
 \item\label{item_Def_KAE_quot} \textbf{$\boldsymbol{\mathcal{K}}$ is $\boldsymbol{\kappa}$-almost everywhere closed under quotients} if the following holds:  whenever $M_\bullet$ is a complex of $R$-modules, $M_\bullet \in \mathcal{N} \prec \mathcal{N}' \prec \mathcal{H}_\theta$, and both $\mathcal{N}$ and $\mathcal{N}'$ have transitive intersection with $\kappa$, then the following implication holds:
\begin{equation}\label{eq_TwoElemSub_Quot}
\Big( M_\bullet \in \mathcal{K}, \ M_\bullet \restriction \mathcal{N} \in \mathcal{K} \text{ and } M_\bullet \restriction \mathcal{N}' \in \mathcal{K} \Big) \implies \ \frac{M_\bullet \restriction \mathcal{N}'}{M_\bullet \restriction \mathcal{N}} \in \mathcal{K}. \tag{*}
\end{equation}
 
 \item \textbf{$\boldsymbol{\mathcal{K}}$ is $\boldsymbol{\kappa}$-almost everywhere closed under transfinite extensions} if the following holds:  whenever $M_\bullet \in \mathcal{K}$ and $\vec{\mathcal{N}} = \langle \mathcal{N}_\xi \ : \ \xi < \eta \rangle$ is a $\subseteq$-continuous and $\subseteq$-increasing sequence of elementary submodels of $\mathcal{H}_\theta$ such that:
 \begin{enumerate}
  \item $M_\bullet \in \mathcal{N}_0$;

  \item Each $\mathcal{N}_\xi$ has transitive intersection with $\kappa$;

  \item $M_\bullet \restriction \mathcal{N}_0 \in \mathcal{K}$; and
  
  \item $\frac{M_\bullet \restriction \mathcal{N}_{\xi+1}}{M_\bullet \restriction \mathcal{N}_{\xi}} \in \mathcal{K}$ whenever $\xi+1 < \eta$, 
  \end{enumerate}
   then $M_\bullet \restriction \bigcup_{\xi < \eta} \mathcal{N}_\xi$ is an element of $\mathcal{K}$.
    \end{enumerate}
    
\noindent We say that ``$\mathcal{K}$ is eventually almost everywhere closed under quotients and transfinite extensions" if these properties hold for all sufficiently large regular $\kappa$.  

For a class $\mathcal{F}$ of modules, the definition of $\kappa$-almost everywhere closure under quotients and transfinite extensions is defined similarly, with the obvious adjustments; e.g., for a module $M$ with $M \in \mathcal{N} \prec \mathcal{N}' \prec \mathcal{H}_\theta$, the requirement \eqref{eq_TwoElemSub_Quot} is replaced by:
\begin{equation}
\Big(M \in \mathcal{F}, \  \langle \mathcal{N} \cap M \rangle \in \mathcal{F}, \text{ and } \langle \mathcal{N}' \cap M \rangle \in \mathcal{F} \Big) \implies \ \frac{\langle \mathcal{N}' \cap M \rangle}{\langle \mathcal{N} \cap M \rangle}  \in \mathcal{F}. \tag{**}
\end{equation}
\end{definition}

Lemma \ref{lem_QuotNecCondition} is then immediate, by the \ref{item_KappaDecon} $\implies$ \ref{item_ElemSub_PAIRS} direction of Theorem \ref{thm_CharacterizeDecon}.

\begin{observation}\label{obs_EntireQuotient}
If $\mathcal{K}$ is $\kappa$-a.e.\ closed under quotients, $M_\bullet \in \mathcal{N} \prec \mathcal{H}_\theta$, $\mathcal{N} \cap \kappa$ is transitive, and $M_\bullet$ and $M_\bullet \restriction \mathcal{N}$ are both in $\mathcal{K}$, then $M_\bullet / \mathcal{N} \in \mathcal{K}$.
\end{observation}
\begin{proof}
Just consider $\mathcal{N}':= \mathcal{H}_\theta$, in which case $M_\bullet \restriction \mathcal{N}'=M_\bullet$.
\end{proof}

The reader may wonder why there was an $\mathcal{N}'$ involved at all in Definition \ref{def_KappaAE}; i.e., why not just require that if $M_\bullet$ and $M_\bullet \restriction \mathcal{N}$ are in $\mathcal{K}$, then $M_\bullet / \mathcal{N}$ is in $\mathcal{K}$?  The role of the (possibly small) $\mathcal{N}'$ in Definition \ref{def_KappaAE} will become apparent in the closure argument in the proof of Theorem \ref{thm_AbstractImplyDec}.  

On the other hand, the reader may also wonder why we include $M_\bullet \in \mathcal{K}$ in the hypotheses of \eqref{eq_TwoElemSub_Quot}; i.e., why not require the following stronger variant of \eqref{eq_TwoElemSub_Quot}?
\begin{equation}\label{eq_NotTheDef}
\Big(  \ M_\bullet \restriction \mathcal{N} \in \mathcal{K} \text{ and } M_\bullet \restriction \mathcal{N}' \in \mathcal{K} \Big) \implies \ \frac{M_\bullet \restriction \mathcal{N}'}{M_\bullet \restriction \mathcal{N}} \in \mathcal{K}.
\end{equation}
The reason is that doing so would result in too weak of a version of Theorem \ref{thm_MainThm}, and would also \emph{not} be good enough to prove our results about Gorenstein Projectivity.  The weaker implication    \eqref{eq_TwoElemSub_Quot} holds for the classes of complexes relevant to Gorenstein Projectivity, while the stronger \eqref{eq_NotTheDef} does not (at least, it is not clear that it holds).  See the proof of Lemma \ref{lem_GorRefLemma} below, where the $P_\bullet$ itself is assumed to be in the relevant class of complexes; this assumption is used in order to apply Lemma \ref{lem_Kaplansky} and the 3-by-3 lemma at a crucial point in that proof.

The ``$\kappa$-almost everywhere" terminology was chosen because of the connection with Shelah's \emph{Stationary Logic}, which is a relatively well-behaved fragment of 2nd Order Logic that uses a quantifier \emph{aa} to express ``almost all" relative to some filter.  The connection is easiest to see in the case $\kappa = \aleph_1$, since elementary submodels of $\mathcal{H}_\theta$ \textbf{always} have transitive intersection with $\omega_1$.\footnote{For $\kappa \ge \omega_2$ things are complicated by the possibility that Chang's Conjecture may hold.}  The class $\mathcal{K}$ is $\aleph_1$-a.e.\ closed under quotients if and only if for every complex $M_\bullet$ and every $\mathcal{H}_\theta$ such that $M_\bullet \in H_\theta$, the structure $(H_\theta,\in, M_\bullet, \mathcal{K} \cap H_\theta)$ satisfies: 
\begin{equation}\label{eq_aa_ClosedQuot}
 \text{aa} Z \ \text{aa} Z'  \left(  \big(M_\bullet \in \mathcal{K}, \ M_\bullet \restriction Z \in \mathcal{K},  \text{ and } M_\bullet \restriction Z' \in \mathcal{K} \big) \ \implies  \frac{M_\bullet \restriction Z'}{M_\bullet \restriction Z} \in \mathcal{K} \right),
\end{equation}
where the \emph{aa} quantifier here refers to the ``strong" club filter on the full powerset of $H_\theta$.\footnote{This is the filter on the full $\wp(H_\theta)$ generated by sets of the form $C_F:=\{ Z \subset H_\theta \ : \ Z \text{ is closed under } F  \}$ where $F: [H_\theta]^{<\omega} \to H_\theta$; see Foreman~\cite{MattHandbook}.  In particular, the $Z$'s and $Z'$'s here are not necessarily required to be countable.  This is not to be confused with the club filter on $\wp_{\omega_1}(H_\theta)$ in the sense of Jech, whose measure one sets concentrate on countable subsets of $H_\theta$; Jech's filter can be viewed as the restriction of the club filter to $\wp_{\omega_1}(H_\theta)$.}  Similarly, part \ref{item_ElemSub} of the statement of Theorem \ref{thm_CharacterizeDecon} can also be expressed in terms of Stationary Logic.

\subsection{Special case:  $\mathfrak{X}$-Gorenstein Projective modules}\label{sec_GorenQuot}

For any class $\mathfrak{X}$ of $R$-modules, an $R$-module $G$ is called \textbf{$\boldsymbol{\mathfrak{X}}$-Gorenstein Projective} if there exists an exact complex $P_\bullet=\big(\xymatrix{ P_n \ar[r]^-{f_n} & P_{n+1} }\big)_{n \in \mathbb{Z}}$ such that:
\begin{itemize}
 \item each $P_n$ is projective;
 \item For all $X \in \mathfrak{X}$, the complex $\text{Hom}_R(P_\bullet,X)$ is exact (we will often express this by saying ``$P_\bullet$ is $\text{Hom}_R(-,\mathfrak{X})$-exact"); and  
 \item $G= \text{ker}(f_0)$.
\end{itemize}

\noindent The class of all exact, $\text{Hom}_R(-,\mathfrak{X})$-exact complexes of projective modules will be denoted $\boldsymbol{\mathcal{K}(\mathfrak{X}\textbf{-}\mathcal{GP})}$, and the class of all $\mathfrak{X}$-Gorenstein Projective modules will be denoted $\boldsymbol{\mathfrak{X}} \textbf{-} \boldsymbol{\mathcal{GP}}$.  If the $\mathfrak{X}$ is not specified, as in $\mathcal{GP}$, it is understood to be the class of all projective modules.  When $\mathfrak{X}$ is the class of all flat modules, $\mathfrak{X}$-$\mathcal{GP}$ is known as the class of \emph{Ding Projective} modules.  Other instances of $\mathfrak{X}$ were considered in \cite{MR2591696}, \cite{bravo2014stable}, and elsewhere.

The main result of this section is Corollary \ref{cor_KGP_KappaAEquottfext}, which says that if $R$ is $<\kappa$-Noetherian, then classes of complexes of the form $\mathcal{K}(\mathfrak{X}\text{-}\mathcal{GP})$ are always closed under transfinite extensions (this is due to Enochs-Iacob-Jenda~\cite{MR2342674}), and are $\kappa$-a.e.\ downward closed under quotients.

\begin{lemma}[Gorenstein Projective Quotient Lemma]\label{lem_GorRefLemma} 
Suppose $\kappa$ is regular, $R$ is $<\kappa$-Noetherian, $\mathfrak{X}$ is any collection of $R$-modules, and $P_\bullet \in \mathcal{K}(\mathfrak{X}\text{-}\mathcal{GP})$.  Suppose $\mathcal{N} \prec \mathcal{H}_\theta$, $\mathcal{N} \cap \kappa$ is transitive, and $R$ and $P_\bullet$ are elements of $\mathcal{N}$. Suppose also that $P_\bullet \restriction \mathcal{N} \in \mathcal{K}(\mathfrak{X}\text{-}\mathcal{GP})$.  Let $G:= \text{ker}(P_0 \to P_1)$.   Then: 
 \begin{enumerate}
  \item\label{item_0thMapPN} $P_\bullet \restriction \mathcal{N}$ witnesses that $\langle \mathcal{N} \cap G \rangle$ is $\mathfrak{X}$-Gorenstein projective.  
  \item\label{item_PfracN} $P_\bullet / \mathcal{N}$ is in $\mathcal{K}(\mathfrak{X}\text{-}\mathcal{GP})$ and witnesses that  $\frac{G}{\langle \mathcal{N} \cap G \rangle}$ is $\mathfrak{X}$-Gorenstein projective.
  \item\label{item_AnotherNprime} If $\mathcal{N}'$ is another elementary submodel such that $\mathcal{N} \prec \mathcal{N}' \prec \mathcal{H}_\theta$, $\mathcal{N}'$ has transitive intersection with $\kappa$, and $P_\bullet \restriction \mathcal{N}' \in \mathcal{K}(\mathfrak{X}\text{-}\mathcal{GP})$, then $\frac{P_\bullet \restriction \mathcal{N}'}{ P_\bullet \restriction \mathcal{N}}$ is in the class $\mathcal{K}(\mathfrak{X}\text{-}\mathcal{GP})$, and witnesses that $\frac{\langle \mathcal{N}' \cap G \rangle}{\langle \mathcal{N} \cap G \rangle}$ is in $\mathfrak{X}$-$\mathcal{GP}$.
 \end{enumerate}
 \end{lemma}
\begin{proof}

Both $P_\bullet \restriction \mathcal{N}$ and $P_\bullet / \mathcal{N}$ are exact, by Lemma \ref{lem_ExactRestrctNoProjective}.
That lemma also shows that the kernels at index 0 of $P_\bullet \restriction \mathcal{N}$ and $P_\bullet / \mathcal{N}$ are, respectively,
\[
\langle \mathcal{N} \cap G \rangle \text{ and } \frac{G}{\langle \mathcal{N} \cap G \rangle}.
\]
Now $P_\bullet \restriction \mathcal{N}$ is in $\mathcal{K}(\mathfrak{X}\text{-}\mathcal{GP})$ by assumption, so this completes the proof of part \eqref{item_0thMapPN}.  To prove part \eqref{item_PfracN}, it remains to show that $P_\bullet / \mathcal{N}$ is in the class $\mathcal{K}(\mathfrak{X}\text{-}\mathcal{GP})$.  It is an exact complex of projective modules, by Lemma \ref{lem_ExactRestrctNoProjective}, so we just have to show that $P_\bullet / \mathcal{N}$ is $\text{Hom}_R(-,\mathfrak{X})$-exact.

Lemma \ref{lem_Kaplansky} implies that for all $n \in \mathbb{Z}$, $\langle \mathcal{N} \cap P_n \rangle$ and $P_n / \langle \mathcal{N} \cap P_n \rangle$ are projective.  Hence, for each $n \in \mathbb{Z}$,
\begin{equation}\label{eq_SplitAt_n}
\xymatrix{
0 \ar[r] & \langle \mathcal{N} \cap P_n \rangle \ar[r]^-{\text{id}_n} & P_n \ar[r]^-{\pi_n} & P_n / \langle \mathcal{N} \cap P_n \rangle \ar[r] & 0 
}
\end{equation}
is \emph{split} exact.\footnote{We are not claiming that $0_\bullet \to P_\bullet \restriction \mathcal{N} \to P_\bullet \to P_\bullet / \mathcal{N} \to 0_\bullet$ splits as a sequence of complexes; i.e.\ it is probably not the case that the splittings of \eqref{eq_SplitAt_n} commute with the $f_n$'s. }  This implies that for all $n \in \mathbb{Z}$ and any module $X$,
\[
\xymatrix{
0 \ar@{<-}[r] & \text{Hom}_R\big( \langle \mathcal{N} \cap P_n \rangle, X \big) \ar@{<-}[r]^-{\text{id}^*_n} & \text{Hom}_R\big( P_n, X \big) \ar@{<-}[r]^-{\pi_n^*} & \text{Hom}_R\big( P_n / \langle \mathcal{N} \cap P_n \rangle ,X \big) \ar@{<-}[r] & 0
}
\]
is (split) exact, and hence that for every module $X$,
\begin{equation}\label{eq_SES_of_Complexes}
\xymatrix{
0_\bullet \ar@{<-}[r] & \text{Hom}_R\big( P_\bullet \restriction \mathcal{N}, X \big) \ar@{<-}[r]^-{\text{id}^*_\bullet} & \text{Hom}_R\big( P_\bullet, X \big) \ar@{<-}[r]^-{\pi^*_\bullet} & \text{Hom}_R\big( P_\bullet / \mathcal{N} ,X\big) \ar@{<-}[r] & 0_\bullet
}
\end{equation}
is an exact sequence of complexes (though not necessarily \emph{split} exact as a sequence of complexes).

Now assume $X$ is any member of $\mathfrak{X}$.  Then:
\begin{itemize}
 \item $\text{Hom}_R(P_\bullet, X)$ is exact by assumption;
 \item $\text{Hom}_R(P_\bullet \restriction \mathcal{N}, X)$ is exact, by the assumption that $P_\bullet \restriction \mathcal{N} \in \mathcal{K}(\mathfrak{X}\text{-}\mathcal{GP})$; and 
 \item  \eqref{eq_SES_of_Complexes} is an exact sequence of complexes. 
\end{itemize}

Then by the 3-by-3 lemma for complexes, $\text{Hom}_R(P_\bullet / \mathcal{N},X)$ is also exact.  This completes the proof of part \eqref{item_PfracN}.  The proof of part \eqref{item_AnotherNprime} is similar, except one instead uses Lemma \ref{lem_Kaplansky} to note 
that $\frac{\langle \mathcal{N}' \cap P_n \rangle}{\langle \mathcal{N} \cap P_n \rangle}$ is projective, and hence the short exact sequence
\[
\xymatrix{
0 \ar[r] & \langle \mathcal{N} \cap P_n \rangle \ar[r]^-{\text{id}} & \langle \mathcal{N}' \cap P_n \rangle \ar[r] & \frac{\langle \mathcal{N}' \cap P_n \rangle}{\langle \mathcal{N} \cap P_n \rangle} \ar[r] & 0
}
\]
splits for each $n \in \mathbb{Z}$.  The rest of the argument is identical to the one above.

\end{proof}

\begin{corollary}\label{cor_KGP_KappaAEquottfext}
Suppose $\kappa$ is regular and uncountable, $R$ is $<\kappa$-Noetherian, and $\mathfrak{X}$ is any class of $R$-modules.  Then the class $\mathcal{K}(\mathfrak{X}\text{-}\mathcal{GP})$ of complexes is closed under (all) transfinite extensions, and is $\kappa$-a.e.\ closed under quotients.
\end{corollary}
\begin{proof}
The class $\mathcal{K}(\mathfrak{X}\text{-}\mathcal{GP})$ is closed under \emph{all} transfinite extensions---not just $\kappa$-a.e.\ closed under transfinite extensions---by Theorem 2.6 of Enochs-Iacob-Jenda~\cite{MR2342674}.  They proved it for $\mathfrak{X} = \{ \text{ projectives } \}$, but the proof goes through for any $\mathfrak{X}$.  The class $\mathcal{K}(\mathfrak{X}$-$\mathcal{GP})$ of complexes is $\kappa$-a.e.\ closed under quotients by part \ref{item_AnotherNprime} of Lemma \ref{lem_GorRefLemma}.  
\end{proof}

\section{$\mathcal{K}$-reflecting elementary submodels}\label{sec_K_ref_elem_subm}

In this section we introduce the key concept that will be used in the proofs of Theorems \ref{thm_Saroch} and \ref{thm_MainThm}:  the notion of a $\mathcal{K}$-reflecting elementary submodel (where $\mathcal{K}$ is a class of complexes).  The main result is Theorem \ref{thm_AbstractImplyDec}, which gives a sufficient condition for a class $\mathcal{K}$ of complexes of modules (and some associated classes of modules) to be $<\kappa$-deconstructible.

Although these results will later be used, along with large cardinals, to prove Theorems \ref{thm_Saroch} and  \ref{thm_MainThm}, we chose to keep this section free of large cardinals, in the hope of being applicable to some still-open questions (e.g., for which rings is $\mathcal{GP}$ \emph{countably} deconstructible; for which rings is $\mathcal{GP} \subseteq \mathcal{GF}$; etc.).  For example, Corollary \ref{cor_KGP_KappaAEquottfext} implies that if $R$ is $\aleph_0$-Noetherian, to show that $\mathcal{GP}$ is $\aleph_0$-deconstructible, it would suffice to show that stationarily many countable elementary submodels of $\mathcal{H}_\theta$ are $\mathcal{GP}$-reflecting.

\begin{definition}\label{def_K_reflect}
Suppose $\mathcal{K}$ is a class of complexes of modules.  If $M_\bullet \in \mathcal{K}$ and $M_\bullet \in \mathcal{N} \prec \mathcal{H}_\theta$, we will say that \textbf{$\boldsymbol{\mathcal{N}}$ is $\boldsymbol{\mathcal{K}}$-reflecting at $\boldsymbol{M_\bullet}$} if $M_\bullet \restriction \mathcal{N} \in \mathcal{K}$.  We will say that $\mathcal{N}$ is \textbf{$\boldsymbol{\mathcal{K}}$-reflecting} if, for every $M_\bullet \in \mathcal{N} \cap \mathcal{K}$, $\mathcal{N}$ is $\mathcal{K}$-reflecting at $M_\bullet$.

In the special case where $\mathcal{K}$ is of the form ``the class of all $\text{Hom}_R(-,\mathfrak{X})$-exact, exact complexes of projective modules" for some class $\mathfrak{X}$, we will sometimes write ``\textbf{$\boldsymbol{\mathfrak{X}}$-$\boldsymbol{\mathcal{GP}}$-reflecting}" instead of ``$\mathcal{K}\big(\mathfrak{X}\text{-}\mathcal{GP}\big)$-reflecting".  This is a slight abuse of terminology because $\mathfrak{X}$-$\mathcal{GP}$ is a class of modules, not a class of complexes.
\end{definition}

The following examples are illustrative.  Suppose $\kappa$ is regular and uncountable, and $R$ is $<\kappa$-Noetherian.
\begin{itemize}
 \item Let $\mathcal{K}$ be the class of all exact complexes of projective $R$-modules.  Then for \emph{any} $\mathcal{N}$ such that $R \in \mathcal{N} \prec \mathcal{H}_\theta$ and $\mathcal{N} \cap \kappa$ is transitive, Lemma \ref{lem_ExactRestrctNoProjective} implies that $\mathcal{N}$ is $\mathcal{K}$-reflecting.  In particular, this holds for $<\kappa$-sized $\mathcal{N}$ such that $\mathcal{N} \cap \kappa \in \kappa$.  Moreover, $\mathcal{K}$ is also $\kappa$-a.e.\ closed under quotients and transfinite extensions.  
 \item On the other hand, while Corollary \ref{cor_KGP_KappaAEquottfext} tells us that (for any $\mathfrak{X}$) the class $\mathcal{K}(\mathfrak{X}$-$\mathcal{GP})$ is closed under transfinite extensions and is $\kappa$-a.e.\ closed under quotients, it is not clear---even for $\mathfrak{X} = \{\text{ projectives} \}$---whether there are \emph{any} $\mathcal{N} \prec \mathcal{H}_\theta$ of size $<\kappa$ that are $\mathcal{K}(\mathfrak{X}$-$\mathcal{GP})$-reflecting.\footnote{Or even whether, for fixed $P_\bullet \in \mathcal{K}(\mathfrak{X}$-$\mathcal{GP})$, there are any $\mathcal{N} \prec \mathcal{H}_\theta$ of size $<\kappa$ that are $\mathcal{K}(\mathfrak{X}$-$\mathcal{GP})$-reflecting at $P_\bullet$. }  We will use large cardinals to remedy this situation in Section \ref{sec_WithHelpOfLC}.  
\end{itemize}

For regular uncountable $\kappa$, $\boldsymbol{\wp_\kappa(H_\theta)}$ refers to the $<\kappa$-sized subsets of $H_\theta$.  We may sometimes write $\boldsymbol{\wp^*_\kappa(H_\theta)}$ to denote the set of $\mathcal{N} \in \wp_\kappa(H_\theta)$ such that $\mathcal{N} \cap \kappa$ is transitive; since $|\mathcal{N}|<\kappa$ and $\kappa$ is regular, this is the same as saying that $\mathcal{N} \cap \kappa \in \kappa$.  For $\kappa = \omega_1$, $\wp_\kappa(H_\theta)$ is essentially (mod clubs) the same as $\wp^*_\kappa(H_\theta)$, but for $\kappa \ge \omega_2$ the two collections can differ significantly (if Chang's Conjecture holds).  A subset $C$ of $\wp^*_\kappa(H_\theta)$ is \textbf{closed and unbounded (club) in $\boldsymbol{\wp^*_\kappa(H_\theta)}$} if $C$ is $\subseteq$-cofinal in $\wp^*_\kappa(H_\theta)$ and closed under $\subseteq$-increasing sequences of length strictly less than $\kappa$.  A set is \textbf{stationary} if it intersects every club.  Kueker's~\cite{MR0457191} following characterization of clubs is convenient; see Theorem 3.6 of Foreman~\cite{MattHandbook} for the formulation given here:  
\begin{theorem}[Kueker]\label{thm_Kueker}
A set $C \subseteq \wp^*_\kappa(H_\theta)$ contains a club if and only if there exists an $F:[H_\theta]^{<\omega} \to H_\theta$ such that
\[
C_F:= \{ z \in \wp^*_\kappa(H_\theta) \ : \  z \text{ is closed under } F    \} \subseteq C.
\]
\end{theorem}

\noindent While the distinction between finite sets and finite sequences is typically unimportant when dealing with structures that satisfy a large fragment of ZFC, the notation $[H_\theta]^{<\omega}$ from the statement of Theorem \ref{thm_Kueker} officially denotes the set of all finite \emph{sequences} from $H_\theta$.

The next theorem says, roughly, that if a class $\mathcal{K}$ of complexes is $\kappa$-a.e.\ closed under quotients and transfinite extensions, and stationarily many elements of $\wp^*_\kappa(H_\theta)$ are $\mathcal{K}$-reflecting, then in fact this is true for club-many.  This, together with an easy induction, yields part \ref{item_ElemSub} of Theorem \ref{thm_CharacterizeDecon} (and hence strong $<\kappa$-deconstructibility).

\begin{theorem}\label{thm_AbstractImplyDec}
Suppose $\kappa$ is regular and uncountable, $R$ is a $<\kappa$-Noetherian ring, and $\mathcal{K}$ is a class of complexes of $R$-modules.  Suppose:

\begin{enumerate}[label=(\Alph*)$_\kappa$]
 \item\label{item_StatManyKreflect} Whenever $M_\bullet \in \mathcal{K}$ and $M_\bullet \in H_\theta$ for some regular $\theta \ge \kappa$, the set
\[
S^{M_\bullet}_{\mathcal{K}\text{-ref}}:= \Big\{ \mathcal{N} \in \wp^*_\kappa(H_\theta) \ : \ \mathcal{N} \text{ is } \mathcal{K} \text{-reflecting at } M_\bullet \text{ (i.e., } M_\bullet \restriction \mathcal{N} \in \mathcal{K} \text{)} \Big\}
\]
is stationary in $\wp^*_\kappa(H_\theta)$.
 \item\label{item_Abstract_3_by_3}  $\mathcal{K}$ is $\kappa$-a.e.\ downward closed under quotients (Definition \ref{def_KappaAE}).

 \item\label{item_K_closed_transfExt} $\mathcal{K}$ is $\kappa$-a.e.\ downward closed under transfinite extensions (Definition \ref{def_KappaAE}).
\end{enumerate}

\noindent Then:
\begin{enumerate}
 \item If $M_\bullet \in \mathcal{K}$, the set $S^{M_\bullet}_{\mathcal{K}\text{-ref}}$ is actually club (not just stationary) in $\wp^*_\kappa(H_\theta)$. 

 \item\label{item_LargeModels} If $M_\bullet \in \mathcal{K}$, then for sufficiently large $\theta$:  whenever $\mathcal{N} \prec (H_\theta,\in,R,M_\bullet,\kappa,\mathcal{K} \cap H_\theta)$ and $\mathcal{N} \cap \kappa$ is transitive, both $M_\bullet \restriction \mathcal{N}$ and $M_\bullet / \mathcal{N}$ are in $\mathcal{K}$.
 
 \item\label{item_DecMOdGK} The class 
 \[
\mathcal{G}(\mathcal{K}):=\big\{ G \ : \ \exists M_\bullet \in \mathcal{K} \ \ G = \text{ker}(M_0 \to M_1)  \big\}
\]
is a strongly $<\kappa$-deconstructible class of modules.

 \item\label{item_DecComplex} $\mathcal{K}$ is a strongly $<\kappa$-deconstructible class of complexes.

\end{enumerate} 
\end{theorem}

\begin{proof}
Fix any $M_\bullet \in \mathcal{K} \cap H_\theta$.  Let $S:= S^{M_\bullet}_{\mathcal{K}\text{-ref}}$, which is stationary in $\wp^*_\kappa(H_\theta)$ by assumption.  Without loss of generality, every $\mathcal{N} \in S$ is an elementary submodel of $\mathcal{H}_\theta$.  To prove that $S$ is in fact a club, we need to show it is closed under $\subseteq$-increasing chains of length strictly less than $\kappa$, and to prove this it suffices to show this is true for such chains that are also $\subseteq$-continuous.  So assume $\eta < \kappa$ and $\langle \mathcal{N}_\xi \ : \ \xi < \eta \rangle$ is a $\subseteq$-increasing and $\subseteq$-continuous sequence of members of $S$, and let $\mathcal{N}_\eta:= \bigcup_{\xi < \eta} \mathcal{N}_\xi$.  Then $\mathcal{N}_\eta \in \wp^*_\kappa(H_\theta)$, and since $M_\bullet \in \mathcal{K}$ and each $\mathcal{N}_\xi$ is in $S$, assumption \ref{item_Abstract_3_by_3} ensures that $\frac{M_\bullet \restriction \mathcal{N}_{\xi+1}}{M_\bullet \restriction \mathcal{N}_\xi}$ is in $\mathcal{K}$ for all $\xi < \eta$
(here the $\mathcal{N}_\xi$ is playing the role of $\mathcal{N}$, and $\mathcal{N}_{\xi+1}$ is playing the role of the $\mathcal{N}'$, from Definition \ref{def_KappaAE}).  So $\langle M_\bullet \restriction \mathcal{N}_\xi \ : \ \xi < \eta \rangle$ is a $\mathcal{K}$-filtration of $M_\bullet \restriction \mathcal{N}_\eta$, and so assumption \ref{item_K_closed_transfExt} ensures that $M_\bullet \restriction \mathcal{N}_\eta $ is in $\mathcal{K}$.

So $S$ in fact is club in $\wp^*_\kappa(H_\theta)$.  By Kueker's Theorem \ref{thm_Kueker}, there is an $F:[H_\theta]^{<\omega} \to H_\theta$ such that whenever $\mathcal{N} \in \wp^*_\kappa(H_\theta)$ and $\mathcal{N}$ is closed under $F$, then $\mathcal{N} \in S$.  Let
\[
\mathfrak{A}_{M_\bullet} := \Big( H_\theta,\in,(F \restriction H_\theta^n)_{n \in \mathbb{N}}, R, M_\bullet \Big),
\]
and observe that elementary substructures of $\mathfrak{A}_{M_\bullet}$ are closed under $F$.  To finish off part \eqref{item_LargeModels}, by Lemma \ref{lem_NewAlgebraLemma} it suffices to prove that whenever $\mathcal{N} \prec \mathfrak{A}_{M_\bullet}$ and $\mathcal{N} \cap \kappa$ is transitive, then both $M_\bullet \restriction \mathcal{N}$ and $M_\bullet / \mathcal{N}$ are in $\mathcal{K}$.  We prove it by induction on $|\mathcal{N}|$.  The base case is when $|\mathcal{N}|<\kappa$:  suppose $\mathcal{N} \prec \mathfrak{A}_{M_\bullet}$, $|\mathcal{N}|<\kappa$, and $\mathcal{N} \cap \kappa$ is transitive.  Then $\mathcal{N} \in \wp^*_\kappa(H_\theta)$ and is closed under $F$, so $M_\bullet \restriction \mathcal{N} \in \mathcal{K}$.  Since $M_\bullet$ and $M_\bullet \restriction \mathcal{N}$ are both in $\mathcal{K}$, the $\kappa$-a.e.\ closure of $\mathcal{K}$ under quotients, implies $M_\bullet / \mathcal{N} \in \mathcal{K}$ (see Observation \ref{obs_EntireQuotient}).

  Now suppose $\mathcal{N} \prec \mathfrak{A}_{M_\bullet}$, $\mathcal{N} \cap \kappa$ is transitive, $|\mathcal{N}| \ge \kappa$, and the induction hypothesis holds for all elementary submodels of size $<|\mathcal{N}|$.  Let $\lambda:= |\mathcal{N}|$.  By a variant of Fact \ref{fact_SmoothChain} there is a $\subseteq$-increasing and $\subseteq$-continuous (but not necessarily $\in$-increasing) sequence $\langle \mathcal{N}_\xi \ : \ \xi < \text{cf}(\lambda) \rangle$ with union $\mathcal{N}$, such that each $\mathcal{N}_\xi$ is of size $<\lambda$, has transitive intersection with $\kappa$, and is elementary in $\mathcal{N}$ (and hence in $\mathfrak{A}_{M_\bullet}$).\footnote{It is possible here that $\mathcal{N}$ may be a ``Chang-type" structure; i.e., $|\mathcal{N}| = \lambda \ge \kappa$ but $\mathcal{N} \cap \kappa \in \kappa$.  But this does not affect the argument; one can just mimic the proof of Fact \ref{fact_SmoothChain} by building a $\subseteq$-increasing and continuous chain of elementary substructures of $\mathcal{N}$, all of which have the same intersection with $\kappa$ that $\mathcal{N}$ has.  This sequence will \textbf{not} be $\in$-increasing, but that is not needed here.}  So, by the induction hypothesis, $M_\bullet \restriction \mathcal{N}_\xi$ and $M_\bullet / \mathcal{N}_\xi$ are in $\mathcal{K}$ for all $\xi < \text{cf}(\lambda)$.  In particular, since $M_\bullet$, $M_\bullet \restriction \mathcal{N}_{\xi+1} $, and $M_\bullet \restriction \mathcal{N}_\xi$ are all members of $\mathcal{K}$, assumption \ref{item_Abstract_3_by_3} ensures that $\frac{M_\bullet \restriction \mathcal{N}_{\xi+1}}{M_\bullet \restriction \mathcal{N}_{\xi}}$ is in $\mathcal{K}$.  Since this holds for all $\xi$, assumption \ref{item_K_closed_transfExt} ensures that $M_\bullet \restriction \mathcal{N} = M_\bullet \restriction \bigcup_\xi \mathcal{N}_\xi$ is in $\mathcal{K}$.  Then, again by assumption \ref{item_Abstract_3_by_3}, $M_\bullet / \mathcal{N}$ is in $\mathcal{K}$.

To prove part \eqref{item_DecMOdGK}, suppose $G \in \mathcal{G}(\mathcal{K})$, as witnessed by some $M_\bullet \in \mathcal{K}$ such that $G = \text{ker}(M_0 \to M_1)$.  Consider any $\mathcal{N} \prec \mathfrak{A}_{M_\bullet}$ such that $\mathcal{N} \cap \kappa$ is transitive.  By part \eqref{item_LargeModels}, both $M_\bullet \restriction \mathcal{N}$ and $M_\bullet / \mathcal{N}$ are in $\mathcal{K}$.  Now, using that $R$ is $<\kappa$-Noetherian and $\mathcal{N} \cap \kappa$ is transitive, we know that the kernel of the 0-th map of $M_\bullet \restriction \mathcal{N}$ is $\langle \mathcal{N} \cap G \rangle$, and that the kernel of the 0-th map of $M_\bullet / \mathcal{N}$ is $G/\langle \mathcal{N} \cap G \rangle$.  Hence, $M_\bullet \restriction \mathcal{N}$ witnesses that $\langle \mathcal{N} \cap G \rangle$ is in $\mathcal{G}(\mathcal{K})$, and $M_\bullet / \mathcal{N}$ witnesses that $G/\langle \mathcal{N} \cap G \rangle$ is in $\mathcal{G}(\mathcal{K})$.  Then by the \ref{item_ElemSub} $\implies$ \ref{item_KappaDecon} direction of Theorem \ref{thm_CharacterizeDecon}, $\mathcal{G}(\mathcal{K})$ is strongly $<\kappa$-deconstructible (the use of $\mathfrak{A}_{M_\bullet}$ is justified by Lemma \ref{lem_NewAlgebraLemma}).  This proves part \eqref{item_DecMOdGK}.

Part \eqref{item_DecComplex} follows from part \eqref{item_LargeModels} together with Theorem \ref{thm_CharacterizeDecon_COMPLEXES} (note this is the only part of the theorem that relies on version Theorem \ref{thm_CharacterizeDecon_COMPLEXES} rather than Theorem \ref{thm_CharacterizeDecon}).  
\end{proof}

\section{Large Cardinals and $\mathcal{K}$-reflecting elementary submodels}\label{sec_WithHelpOfLC}

Theorem \ref{thm_AbstractImplyDec} listed three hypotheses on a class $\mathcal{K}$ of complexes that together guarantee the $<\kappa$-deconstructibility of $\mathcal{K}$ (and of certain associated classes of modules).  Corollary \ref{cor_KGP_KappaAEquottfext} showed that hypotheses \ref{item_Abstract_3_by_3} and \ref{item_K_closed_transfExt} \emph{always} hold for classes of the form $\mathcal{K}(\mathfrak{X}\text{-}\mathcal{GP})$---i.e., those complexes associated with $\mathfrak{X}$-Gorenstein projectivity---as long as the ring is $<\kappa$-Noetherian.  However, arranging that the hypothesis \ref{item_StatManyKreflect} of Theorem \ref{thm_AbstractImplyDec} holds (for some $\kappa$) seems to be tricky; we use large cardinals to guarantee it.

\subsection{Formula reflection and transitive collapses}

We will work with transitive collapses of elementary submodels of $\mathcal{H}_\theta$.  Note that $\mathcal{H}_\theta$ is extensional, so in particular, any elementary submodel of $\mathcal{H}_\theta$ is extensional (and wellfounded of course), and hence has a transitive collapse (see Jech~\cite{MR1940513}).

\begin{convention}\label{conv_MostCollapse}
If $\mathcal{N} \prec \mathcal{H}_\theta$, then $\mathcal{H}_{\mathcal{N}}$ will denote the transitive collapse of $\mathcal{N}$, and
\[
\sigma_{\mathcal{N}}: \mathcal{H}_{\mathcal{N}} \to_{\text{iso}} \mathcal{N} \prec \mathcal{H}_\theta
\]
will denote the inverse of the Mostowski collapsing map; note that $\sigma_{\mathcal{N}}$ is an elementary embedding from $\mathcal{H}_{\mathcal{N}}$ to $\mathcal{H}_\theta$.  For elements $b$ in the range of $\sigma_{\mathcal{N}}$---i.e., for $b \in \mathcal{N}$---we will often write $\bar{b}$ for $\sigma_{\mathcal{N}}^{-1}(b)$.
\end{convention}

Most of our results so far have dealt with situations where $R$ is $<\kappa$-Noetherian and $\mathcal{N} \prec \mathcal{H}_\theta$ has transitive intersection with $\kappa$.  From now on, however, we will typically make the stronger assumption that $R \subset \mathcal{N}$, as in the following lemma:

\begin{lemma}\label{lem_TrCollapseIsoComplex}
Suppose $R$ is an element \textbf{and subset} of $\mathcal{N}$, and $\mathcal{N} \prec \mathcal{H}_\theta$.  Then for any complex $M_\bullet$ of $R$-modules such that $M_\bullet \in \mathcal{N}$, the complex $M_\bullet \restriction \mathcal{N}$ is isomorphic (as a complex) to $\overline{M}_\bullet = \sigma_{\mathcal{N}}^{-1}(M_\bullet)$.
\end{lemma}
\begin{proof}
Say $M_\bullet = \Big( \xymatrix{M_n \ar[r]^-{f_n} & M_{n+1}} \Big)_{n \in \mathbb{Z}}$.  By elementarity of $\sigma_{\mathcal{N}}$, $\mathcal{H}_{\mathcal{N}} \models$ ``$\overline{M}_\bullet$ is a complex of $\overline{R}$-modules", and this is easily upward absolute to $V$.  Since $R \in \mathcal{N}$ and $R \subset \mathcal{N}$, $\sigma_{\mathcal{N}} \restriction \overline{R}$ is a ring isomorphism from $\overline{R}$ onto $R$.  So we can view $\overline{M}_\bullet$ as a complex of $R$-modules.\footnote{In fact, $R \subset \mathcal{N}$ and $R \in \mathcal{N}$ imply that $|R|$ is in the transitive part of $\mathcal{N}$; so we can without loss of generality assume that $\sigma_{\mathcal{N}}$ fixes $R$.} Also, since $R \subset \mathcal{N}$, $\mathcal{N} \cap M_n$ is already closed under scalar multiplication, so $\langle \mathcal{N} \cap M_n \rangle = \mathcal{N} \cap M_n$ for all $n \in \mathbb{Z}$.  It follows that for each $n \in \mathbb{Z}$, 
\[
\sigma_n:= \sigma_{\mathcal{N}} \restriction \overline{M}_n
\] 
is an isomorphism from $\overline{M}_n$ to $\mathcal{N} \cap M_n = \langle \mathcal{N} \cap M_n \rangle$.  Also, by elementarity of $\sigma_{\mathcal{N}}$,
\[
\bar{f}_n = \sigma_{n+1}^{-1} \circ \  \big( f_n \restriction (\mathcal{N} \cap M_n) \big) \ \circ \ \sigma_n
\]
for all $n \in \mathbb{Z}$.  So $\sigma_\bullet$ is an isomorphism from $\overline{M}_\bullet$ onto $M_\bullet \restriction \mathcal{N}$. 
\end{proof}

So by Lemma \ref{lem_TrCollapseIsoComplex}, if $R \cup \{ R \} \subset \mathcal{N} \prec \mathcal{H}_\theta$ and $\mathcal{K}$ is an isomorphism-closed class of complexes, asking whether $\mathcal{N}$ is $\mathcal{K}$-reflecting at some $M_\bullet \in \mathcal{N} \cap \mathcal{K}$ is the same as asking whether $\overline{M}_\bullet$ (the image of $M_\bullet$ under the transitive collapse of $\mathcal{N}$) is also in $\mathcal{K}$.  We will see in Corollary \ref{cor_FormulaSpecificRef} below that \emph{if} the parameters used in the definition of $\mathcal{K}$ aren't moved by the transitive collapsing map of $\mathcal{N}$, and \emph{if} both $\mathcal{H}_\theta$ and $\mathcal{H}_{\mathcal{N}}$ are in some sense ``correct" (from the point of view of the universe $V$) about membership in $\mathcal{K}$, then $\mathcal{N}$ will be $\mathcal{K}$-reflecting.  We first introduce a convenient definition:

\begin{definition}\label{def_FormulaReflect}
Consider a fixed first-order formula $\phi(v_1, \dots, v_k)$ in the language of set theory.
\begin{itemize}
 \item Given a transitive set $H$, we will say that \textbf{$\boldsymbol{(H,\in)}$ reflects the formula $\boldsymbol{\phi}$} if for all $a_1, \dots, a_k \in H$, $\phi(a_1,\dots,a_k)$ holds (in the universe) if and only if 
\[
(H,\in) \models \phi(\vec{a}).
\]
\textbf{Note:}  for a fixed $\phi$, ``$(H,\in)$ reflects the formula $\phi$" is first order expressible.\footnote{In fact, for fixed metamathematical natural number $n$, ``$(H,\in)$ reflects all $\Sigma_n$ formulas" is first order expressible.}  

\item If $\mathcal{N}$ is an elementary submodel of some $\mathcal{H}_\theta$, we will say that \textbf{$\boldsymbol{\mathcal{N}}$ transitively reflects the formula $\boldsymbol{\phi}$} if $\mathcal{H}_{\mathcal{N}}$ (the transitive collapse of $\mathcal{N}$) reflects the formula $\phi$.  

\item For a regular uncountable cardinal $\kappa$, we will say that \textbf{$\boldsymbol{\kappa}$ has $\boldsymbol{\phi}$-transitively reflecting models} if there are unboundedly many regular $\lambda$ such that:
\begin{enumerate}
 \item $\mathcal{H}_\lambda$ reflects the formula $\phi$; and
 
 \item The set 
 \[
T_{\phi,\kappa,\lambda}:= \big\{ \mathcal{N} \in \wp^*_\kappa(H_\lambda) \ : \ \mathcal{N} \text{ transitively reflects the formula } \phi   \big\}
 \]
 is stationary in $\wp_\kappa(H_\lambda)$.   
\end{enumerate}

\end{itemize}
\end{definition}

\begin{corollary}\label{cor_FormulaSpecificRef}
Suppose $\kappa$ is regular and uncountable, $R$ is a ring of size $<\kappa$, $\mathcal{K}$ is an isomorphism-closed class of complexes of $R$-modules, and $\mathcal{K}$ is (set-theoretically) definable from parameters in $H_\kappa$; i.e.\ there is some first order formula $\phi_{\mathcal{K}}$ in the language of set theory, and some fixed $p_1, \dots, p_k \in H_\kappa$, such that
\[
\mathcal{K} = \big\{ x \ : \ \phi_{\mathcal{K}}(x,\vec{p})  \big\}.
\]
Suppose $\lambda \ge \kappa$, $\mathcal{H}_\lambda$ reflects the formula $\phi_{\mathcal{K}}$, and $\mathcal{N} \prec \mathcal{H}_\lambda$ is such that $\{ R, p_1, \dots, p_k \} \subset \mathcal{N} \prec \mathcal{H}_\lambda$, $\mathcal{N} \cap \kappa$ is transitive, and $\mathcal{N}$ transitively reflects the formula $\phi_{\mathcal{K}}$.

Then $\mathcal{N}$ is $\mathcal{K}$-reflecting (in the sense of Definition \ref{def_K_reflect}).
\end{corollary}
\begin{proof}
Fix any $M_\bullet \in \mathcal{N} \cap \mathcal{K}$.  By the definability assumption on $\mathcal{K}$, $(V,\in) \models \phi_{\mathcal{K}}(M_\bullet, \vec{p})$.  By the assumption that $\mathcal{H}_\lambda$ reflects the formula $\phi_{\mathcal{K}}$, $\mathcal{H}_\lambda \models \phi_{\mathcal{K}}(M_\bullet, \vec{p})$.  By elementarity of $\sigma_{\mathcal{N}}$, 
\begin{equation*}
\mathcal{H}_{\mathcal{N}} \models \phi_{\mathcal{K}}\big( \overline{M}_\bullet, \bar{p}_1, \dots, \bar{p}_k \big).
\end{equation*}
Since each $p_i$ is in $\mathcal{N} \cap H_\kappa$ and $\mathcal{N} \cap \kappa$ is transitive, it follows that $\sigma_{\mathcal{N}}$ fixes each $p_i$; so
\begin{equation*}
\mathcal{H}_{\mathcal{N}} \models \phi_{\mathcal{K}}\big( \overline{M}_\bullet, p_1, \dots, p_k \big)
\end{equation*}
Since $\mathcal{N}$ transitively reflects the formula $\phi_{\mathcal{K}}$---i.e., since $\mathcal{H}_{\mathcal{N}}$ reflects $\phi_{\mathcal{K}}$---we have
\[
(V,\in) \models \phi_{\mathcal{K}}(\overline{M}_\bullet, p_1, \dots, p_k).
\]
Hence, by the assumption regarding the definability of $\mathcal{K}$, $\overline{M}_\bullet$ is in $\mathcal{K}$.  Now $R \in \mathcal{N}$, $|R|<\kappa$, and $\mathcal{N} \cap \kappa$ is transitive, so by Fact \ref{fact_BasicElemSub}, $R$ is a subset of $\mathcal{N}$.   Then Lemma \ref{lem_TrCollapseIsoComplex} ensures that $\overline{M}_\bullet$ is isomorphic (as a complex of $R$-modules) to $M_\bullet \restriction \mathcal{N}$.  So by closure of $\mathcal{K}$ under isomorphism, $M_\bullet \restriction \mathcal{N} \in \mathcal{K}$.
\end{proof}

\begin{theorem}\label{thm_FormulaThm_implyDec}
Suppose $\kappa$ is regular and uncountable, $R$ is a ring if size $<\kappa$, and $\mathcal{K}$ is a class of complexes of $R$-modules that is set-theoretically definable from parameters in $H_\kappa$; say
\[
\mathcal{K} = \big\{ x \ : \ \phi_{\mathcal{K}}(x,p_1,\dots, p_k)  \big\}
\]
where $\phi_{\mathcal{K}}$ is a formula in the language of set theory and $p_1, \dots, p_k$ are in $H_\kappa$.  Suppose that:
\begin{enumerate}
 \item $\kappa$ has $\phi_{\mathcal{K}}$-transitively reflecting models (as in Definition \ref{def_FormulaReflect});

 \item $\mathcal{K}$ is $\kappa$-a.e.\ closed under quotients; and
 \item $\mathcal{K}$ is $\kappa$-a.e.\ closed under transfinite extensions.
\end{enumerate}

Then:
\begin{itemize}
 \item $\Big\{ G \ : \ \exists M_\bullet \in \mathcal{K} \ \ G = \text{ker}(M_0 \to M_1) \Big\}$ is a strongly $<\kappa$-deconstructible class of modules.

 \item  $\mathcal{K}$ is a strongly $<\kappa$-deconstructible class of complexes 
\end{itemize}
\end{theorem}
\begin{proof}
This follows from Theorem \ref{thm_AbstractImplyDec} and Corollary \ref{cor_FormulaSpecificRef}.  We remark that the first bullet in the conclusion does \emph{not} rely on Theorem \ref{thm_CharacterizeDecon_COMPLEXES}, but instead relies (indirectly) on Theorem \ref{thm_CharacterizeDecon}.
\end{proof}

\subsection{Proof of \v{S}aroch's Theorem \ref{thm_Saroch} from supercompacts}\label{sec_SupercompSaroch}

\begin{definition}[Viale~\cite{Viale_GuessingModel}]\label{def_0_guess}
An elementary submodel $\mathcal{N}$ of $\mathcal{H}_\theta$ is called a \textbf{0-guessing set in  $\boldsymbol{\wp^*_{\kappa}}(H_\theta)$} if $\kappa \in \mathcal{N}$, $|\mathcal{N}|<\kappa$, $\mathcal{N} \cap \kappa$ is transitive, $\lambda:= \text{otp}(\mathcal{N} \cap \theta)$ is a cardinal, and the transitive collapse of $\mathcal{N}$ is $H_\lambda$.
\end{definition}

\begin{fact}\label{fact_Supercompact_0_guess}
If $\kappa$ is a supercompact cardinal, then for all cardinals $\theta > \kappa$, the set of 0-guessing sets in $\wp^*_\kappa(H_\theta)$ is stationary.  I.e., for all  expansions $\mathfrak{A}$ of $\mathcal{H}_\theta$ in a countable signature, there exists an $\mathcal{N} \prec \mathfrak{A}$ such that $\mathcal{N}$ is a 0-guessing set in $\wp^*_\kappa(H_\theta)$.

\end{fact}
\begin{proof}
Let $j: V \to M$ be an elementary embedding with critical point $\kappa$ such that $M$ is closed under $|H_\theta|$-length sequences and $j(\kappa) > |H_\theta|$.  Then $j[H_\theta]$ is an element of $M$, is a $<j(\kappa)$-sized elementary submodel of $j(\mathfrak{A})$ whose intesection with $j(\kappa)$ is an ordinal (namely, $\kappa$), and whose transitive collapse is exactly $H_\theta^V = H_\theta^M$.  So $M$ believes that $j[H_\theta]$ is a 0-guessing elementary submodel of $j(\mathfrak{A})$ that lies in $\wp^*_{j(\kappa)}(j(H_\theta))$.  By elementarity of $j$, there is a 0-guessing elementary submodel of $\mathfrak{A}$ that lies in $\wp^*_\kappa(H_\theta)$.
\end{proof}

Although we will not use it here, work of Viale~\cite{Viale_GuessingModel} and Magidor~\cite{MR0327518} shows that the converse of Fact \ref{fact_Supercompact_0_guess} also holds.

We now proceed to prove Theorem \ref{thm_Saroch} assuming a proper class of supercompacts.  Let $\phi(x,R)$ abbreviate the statement ``$x$ is an exact and $\text{Hom}_R(-,\text{Proj})$-exact complex of projective $R$-modules".  Let $R$ be any ring, and let $\kappa$ be a supercompact cardinal with $R \in H_\kappa$.  Let 
\[
\mathcal{K}(\mathcal{GP}_R) = \big\{ x \ : \ \phi(x,R) \big\}.
\]
I.e., $\mathcal{K}(\mathcal{GP}_R)$ is the class of all exact, $\text{Hom}_R(-,\text{Proj})$-exact complexes of projective $R$-modules.  By Corollary \ref{cor_KGP_KappaAEquottfext}, $\mathcal{K}(\mathcal{GP}_R)$ is $\kappa$-a.e.\ downward closed under quotients and transfinite extensions.  Now consider any regular $\lambda > \kappa$.  By Fact \ref{fact_Supercompact_0_guess}, there are stationarily many 0-guessing $\mathcal{N} \in \wp^*_\kappa(H_\lambda)$.  Fix such an $\mathcal{N}$; then $\mu:= \text{otp}(\mathcal{N} \cap \lambda)$ is a cardinal and $\mathcal{H}_{\mathcal{N}} = \mathcal{H}_\mu$.  By Corollary \ref{cor_WitnessNonExact}, both $\mathcal{H}_\lambda$ and $\mathcal{H}_\mu = \mathcal{H}_{\mathcal{N}}$ reflect the formula $\phi$ (so $\mathcal{N}$ transitively reflects the formula $\phi$).  So $\kappa$ has $\phi$-transitively reflecting models.  We have verified all the assumptions of Theorem \ref{thm_FormulaThm_implyDec}, which yields the desired decontructibility.

\subsection{Proof of Theorem \ref{thm_MainThm}}\label{sec_PfBigThms}

The proof of Theorem \ref{thm_MainThm} is almost identical to the proof of Theorem \ref{thm_Saroch} given above.  The main difference is that, for a general formula $\phi$ defining a class of complexes, it may not be the case that $\phi$ is ZFC-provably absolute between the universe of sets, and models of the form $\mathcal{H}_\mu$.  In other words, we may not always have a ZFC-provable analogue of Corollary \ref{cor_WitnessNonExact} to work with (though in many cases we do; see Section \ref{sec_Remarks_LC}).

\begin{remark}\label{rem_MetaMath}
We briefly address some metamathematical issues surrounding the statement and proof of Theorem \ref{thm_MainThm}.  Since we cannot quantify over classes in first order set theory, one should view these results as metamathematical statements about the consistency of a certain theory.  Vop\v{e}nka's Principle itself, as usually construed, is a scheme of first order sentences, rather than a single one; see \cite{MR2805294} for details.  We focus on part \ref{item_XGP_Decon} of the theorem, but similar comments apply to the other parts.  For each first order formula $\phi(u,R,p_1, \dots , p_k)$ in the language of set theory, let $\mathfrak{X}_{\phi,R,\vec{p}}$ denote the class
\[
\Big\{ x \ : \ \phi(x,R,p_1,\dots,p_k) \Big\}.
\]
Let $\phi_{\mathcal{GP}}(R,p_1,\dots, p_k)$ be the first order formula expressing that ``if $R$ is a ring and every member of $\mathfrak{X}_{\phi,R,\vec{p}}$ is an $R$-module, then the class of $\mathfrak{X}_{\phi,R,\vec{p}}$-Gorenstein Projective modules is deconstructible".  The proper class notations here---in particular, the use of $\mathfrak{X}_{\phi,R,\vec{p}}$---are just for convenience; the formula $\phi_{\mathcal{GP}}$ can be expressed in a first order manner.  Then part \ref{item_XGP_Decon} of the theorem is really saying:  in any model of ZFC that satisfies Vop\v{e}nka's Principle (scheme), the model also satisfies the first order theory consisting of ZFC together with, for each first order formula $\phi$ in the language of set theory, the following sentence:
\[
\forall R \ \forall p_1 \dots p_k \ \phi_{\mathcal{GP}}(R,p_1,\dots,p_k).
\]
\end{remark}

We are now ready to prove Theorem \ref{thm_MainThm}.  Vop\v{e}nka's Principle (VP) implies the following scheme (the proof that VP implies this scheme is relegated to the Appendix):
\begin{LC_Assumption}\label{lc_PhiRef}
For every formula $\phi$ in the language of set theory, there are unboundedly many $\kappa$ such that $\kappa$ has $\phi$-transitively reflecting models.
\end{LC_Assumption}

First we prove part \ref{item_MaxDecon} of Theorem \ref{thm_MainThm}.  Suppose $\mathcal{K}$ is an isomorphism-closed class of complexes of $R$-modules; say $\mathcal{K}$ is definable via the formula $\phi_{\mathcal{K}}$ and parameters $p_1,\dots, p_k$; i.e.
\[
\mathcal{K} =  \Big\{ x \ : \ \phi_{\mathcal{K}}(x,p_1,\dots,p_k) \Big\}.
\]
By assumption, there is some $\mu_{\mathcal{K}}$ such that for all regular $\kappa \ge \mu_{\mathcal{K}}$, $\mathcal{K}$ is $\kappa$-a.e.\ closed under quotients and transfinite extensions.  By the large cardinal assumption, there is a regular $\kappa$ such that $\kappa \ge \mu_{\mathcal{K}}$, 
\[
\{R,p_1,\dots,p_k \} \subset H_\kappa,
\]
 and $\kappa$ has $\phi_{\mathcal{K}}$-transitively reflecting models.  Then by Theorem \ref{thm_FormulaThm_implyDec}, $\mathcal{K}$ is a strongly $<\kappa$-deconstructible class of complexes.

Part \ref{item_XGP_Decon} of Theorem \ref{thm_MainThm} actually follows from part \ref{item_MaxDecon}, as follows.  Let $R$ be any ring, and $\mathfrak{X}$ be any (definable) class of $R$-modules; then there is a formula $\psi$ in the language of set theory, and some fixed parameter $p$, such that
\[
\mathfrak{X} = \big\{ x \ : \ (V,\in) \models \psi(x,p) \big\}.
\]
It follows that the class $\mathcal{K}(\mathfrak{X}\text{-}\mathcal{GP})$ of exact, $\text{Hom}_R(-,\mathfrak{X})$-exact complexes of projective $R$-modules is also definable from the parameter $p$; i.e., there is some formula $\phi_{\mathcal{K}(\mathfrak{X}\text{-}\mathcal{GP})}$ such that
\[
\mathcal{K}(\mathfrak{X}\text{-}\mathcal{GP}) = \{ c \ : \ \phi_{\mathcal{K}(\mathfrak{X}\text{-}\mathcal{GP})}(c,p) \}.
\]
By Corollary \ref{cor_KGP_KappaAEquottfext}, $\mathcal{K}(\mathfrak{X}\text{-}\mathcal{GP})$ is closed under transfinite extensions, and is eventually almost everywhere closed under quotients.  So by part \ref{item_MaxDecon}, $\mathcal{K}(\mathfrak{X}\text{-}\mathcal{GP})$ is deconstructible.  Then the class $\mathcal{K}(\mathfrak{X}\text{-}\mathcal{GP})$ of complexes is strongly $<\kappa$-deconstructible for some regular $\kappa > |R|$.  This implies that the class $\mathfrak{X}$-$\mathcal{GP}$ of modules is also strongly $<\kappa$-deconstructible, for the following reason.  Suppose $G \in \mathfrak{X}$-$\mathcal{GP}$; then there is some $P_\bullet \in \mathcal{K}(\mathfrak{X}\text{-}\mathcal{GP})$ such that $G = \text{ker}(P_0 \to P_1)$.  Then if $\langle P_\bullet^\xi \ : \ \xi < \eta \rangle$ is a filtration of $P_\bullet$ such that for all $\xi$, $P_\bullet^{\xi+1}/P_\bullet^\xi$ is in $\mathcal{K}(\mathfrak{X}\text{-}\mathcal{GP})$ with strongly $<\kappa$-presented modules at all indices $n \in \mathbb{Z}$, then
\[
\big\langle  \text{ker}(P_0^{\xi} \to P_1^\xi)  \ : \   \xi < \eta \big\rangle
\]
is an $\big(\mathfrak{X}$-$\mathcal{GP}\big)^{<\kappa}_s$-filtration of $G$.

\begin{remark}
Since the proof of part \ref{item_XGP_Decon} just given relies indirectly on Theorem \ref{thm_CharacterizeDecon_COMPLEXES}---for which we did not provide a full proof (though it is very similar to the proof of Theorem \ref{thm_CharacterizeDecon})---we briefly describe how one can prove part \ref{item_XGP_Decon} in a way that does not rely on Theorem \ref{thm_CharacterizeDecon_COMPLEXES}.  This proof closely mimics the proof of Theorem \ref{thm_Saroch} given in Section \ref{sec_SupercompSaroch}.

In this proof we use the large cardinal assumption \ref{lc_PhiRef} directly to find a $\kappa$ that has $\phi_{\mathcal{K}(\mathfrak{X}\text{-}\mathcal{GP})}$-transitively reflecting models, and that also is large enough to witness the eventual almost everywhere closure under quotients of the class $\mathcal{K}(\mathfrak{X}\text{-}\mathcal{GP})$ of complexes (which is possible by Corollary \ref{cor_KGP_KappaAEquottfext}).  Then by Theorem \ref{thm_FormulaThm_implyDec}, 
\[
\Big\{  G \ : \ \exists P_\bullet \in \mathcal{K}(\mathfrak{X}\text{-}\mathcal{GP}) \ \ G = \text{ker}(P_0 \to P_1)   \Big\}
\]
is strongly $<\kappa$-deconstructible; this class is just $\mathfrak{X}$-$\mathcal{GP}$.
\end{remark}

Before proving part \ref{item_MaxKaplansky} of Theorem \ref{thm_MainThm}, we need to define what a \emph{Kaplansky Class} is; this notion was introduced in Enochs and L\'{o}pez-Ramos~\cite{MR1926201}, though our definition is closer to Definition 10.1 of G\"obel-Trlifaj~\cite{MR2985554}.  Given a cardinal $\kappa$, a class $\mathcal{K}$ of complexes of modules is called a \textbf{$\boldsymbol{<\kappa}$-Kaplanky Class (of complexes)} if for all $M_\bullet \in \mathcal{K}$ and all sequences $\langle X_n \ : \ n  \in \mathbb{Z} \rangle$ such that $X_n \subset M_n$ and $|X_n|<\kappa$ for all $n \in \mathbb{Z}$, there exists a subcomplex $N_\bullet \subset M_\bullet$ such that $N_\bullet \in \mathcal{K}$, $M_\bullet / N_\bullet \in \mathcal{K}$, and for all $n \in \mathbb{Z}$, $X_n \subseteq N_n$ and $N_n$ is a $<\kappa$-presented module.  A class is a \textbf{Kaplansky Class} if it is a $<\kappa$-Kaplansky Class for some cardinal $\kappa$.

Now back to the proof of part \ref{item_MaxKaplansky} of Theorem \ref{thm_MainThm}.  Assume that $\mathcal{K}$ is eventually a.e.\ closed under quotients; so there is a $\mu_{\mathcal{K}}$ such that for all regular $\kappa \ge \mu_{\mathcal{K}}$, $\mathcal{K}$ is $\kappa$-a.e.\ closed under quotients (but we do \emph{not} assume $\mathcal{K}$ is closed under transfinite extensions).  Say
\[
\mathcal{K} = \big\{ c \ : \ \phi_{\mathcal{K}}(c,p) \big\}.
\]
By the large cardinal assumption \ref{lc_PhiRef}, there is a $\kappa$ such that $R,p \in H_\kappa$, $\kappa \ge \mu_{\mathcal{K}}$, and $\kappa$ has $\phi_{\mathcal{K}}$-transitively reflecting models.  Fix $M_\bullet \in \mathcal{K}$, and fix some sequence $\vec{X} = \langle X_n \ : \ n \in \mathbb{Z} \rangle$ such that $X_n$ is a $<\kappa$-sized subset of $M_n$ for all $n \in \mathbb{Z}$.  Since $\kappa$ has $\phi_{\mathcal{K}}$-transitively reflecting models, there is a regular $\lambda > \kappa$ such that $M_\bullet \in H_\lambda$, $H_\lambda$ reflects the formula $\phi_{\mathcal{K}}$, and there are stationarily many $\mathcal{N} \in \wp^*_\kappa(H_\lambda)$ that transitively reflect the formula $\phi_{\mathcal{K}}$.  Then there is such an $\mathcal{N}$ with $M_\bullet, \vec{X} \in \mathcal{N}$.  Corollary \ref{cor_FormulaSpecificRef} implies that $M_\bullet \restriction \mathcal{N}$ is in $\mathcal{K}$, and the $\kappa$-a.e.\ closure under quotients then implies that $M_\bullet / \mathcal{N}$ is in $\mathcal{K}$.  Since $\vec{X} \in \mathcal{N}$, it follows that each $X_n$ is a $<\kappa$-sized element of $\mathcal{N}$, so by Fact \ref{fact_BasicElemSub}, each $X_n$ is a subset of $\mathcal{N} \cap M_n$.  And $\langle \mathcal{N} \cap M_n \rangle = \mathcal{N} \cap M_n$ is $|\mathcal{N}|$-presented by Lemma \ref{lem_StrongGen}; in particular, $<\kappa$-presented.  So $M_\bullet \restriction \mathcal{N}$ is the desired subcomplex of $M_\bullet$.

\subsection{On the large cardinal assumption of Theorem \ref{thm_MainThm}}\label{sec_Remarks_LC}

Vop\v{e}nka's Principle allowed us to get, in part \ref{item_XGP_Decon} of Theorem \ref{thm_MainThm}, the deconstructibility of \emph{all} classes of the form $\mathfrak{X}$-$\mathcal{GP}$, without having to delve into the complexity or absoluteness properties of the particular formula defining the class of $\text{Hom}_R(-,\mathfrak{X})$-exact complexes.

However, for many particular, commonly-used instances of $\mathfrak{X}$, to get the deconstructibility of $\mathfrak{X}$-$\mathcal{GP}$, it suffices to assume ``mere" supercompactness (which is weaker in consistency strength than Vop\v{e}nka's Principle).  For example:
\begin{enumerate}[label=(\Alph*)]
 \item In our proof of \v{S}aroch's Theorem in Section \ref{sec_SupercompSaroch}, a supercompact above the size of the ring sufficed to get deconstructiblity of $\mathcal{GP}$.  In that proof, ``mere" 0-guessing models (as opposed to guessing models that transitively reflected more formulas) sufficed because of the ZFC-provable Corollary \ref{cor_WitnessNonExact}.
 \item   One can also get by with a supercompact when $\mathfrak{X}$ is the class of flat modules.  I.e., if $|R|<\kappa$ and $\kappa$ is supercompact, then the class $\mathfrak{X}$-$\mathcal{GP}_R$---also known as the \emph{Ding Projective} $R$-modules---is strongly $<\kappa$-deconstructible.  This requires an analysis similar to the one in Corollary \ref{cor_WitnessNonExact}, and ultimately relies on certain ZFC-provable absoluteness between the universe of sets and the $\mathcal{H}_\lambda$'s regarding flat modules (this analysis is closely related to the fact, proved by Enochs, that the class of flat $R$-modules is a $<|R|^+$-Kaplansky class).  Such technical analysis of the reflection properties of flatness are unnecessary if one is willing to throw caution to the wind and just assume Vop\v{e}nka's Principle.
\end{enumerate}

\section{Open questions}\label{sec_Questions}

\begin{question}
Are any of the conclusions of Theorem \ref{thm_Saroch} or \ref{thm_MainThm} provable in ZFC alone?  \end{question}

\begin{question}
Do any of the conclusions of Theorem \ref{thm_Saroch} or \ref{thm_MainThm} have large cardinal consistency strength? 
\end{question}

\appendix

\section{Vop\v{e}nka's Principle}\label{sec_APPENDIX_VP}

Here we point out why the large cardinal assumption \eqref{lc_PhiRef} follows from Vop\v{e}nka's Principle (VP).    We will make use of the Levy hierarchy of formulas.  For a (meta-mathematical) natural number $n$ and a transitive set $H$, ``$(H,\in) \prec_{\Sigma_n} (V,\in)$" is expressible in the language of set theory (and is in fact $\Sigma_n$-expressible for $n \ge 1$, see Kanamori~\cite{MR1994835}).  Let $C(n)$ denote the class of cardinals $\lambda$ such that $H_\lambda = V_\lambda$ and
\[
(H_\lambda,\in) \prec_{\Sigma_n} (V,\in).
\]
For fixed $n$, $C(n)$ is a definable, closed unbounded class of cardinals.  As in Bagaria et al.~\cite{MR3323199}, a cardinal $\kappa$ is called \textbf{C(n)-extendible} if for all $\lambda \in C(n)$ above $\kappa$, there exists some $\lambda'$ and some 
\[
j: \mathcal{H}_\lambda \to \mathcal{H}_{\lambda'}
\]
such that:
\begin{enumerate}
 \item $\lambda' \in C(n)$;
 \item $j$ is an elementary embedding;
 \item $\text{crit}(j) = \kappa$ and $\lambda < j(\kappa)$; and
 \item Both $\kappa$ and $j(\kappa)$ are in $C(n)$ (this will actually not be used).
\end{enumerate}

Bagaria et al.~\cite{MR3323199} proved that VP is equivalent to the following scheme: for every (meta-mathematical) natural number $n$, there is a proper class of $C(n)$-extendible cardinals. We use this characterization of VP to show:
\begin{lemma}\label{lem_VP_Sigman}
Assume VP.  Then for each $n \ge 1$, there are proper-class many $\kappa$ such that, for proper-class many $\lambda>\kappa$, the following holds:
\begin{enumerate}
 \item $\mathcal{H}_\lambda \prec_{\Sigma_n} (V,\in)$;
 \item\label{item_StatMany_n_ref} There are stationarily many $\mathcal{N} \in \wp^*_\kappa(H_\lambda)$ such that the transitive collapse of $\mathcal{N}$ is a $\Sigma_n$-elementary substructure of the universe.
\end{enumerate}
\end{lemma}

\noindent Clearly the conclusion of the lemma implies the large cardinal assumption \eqref{lc_PhiRef}, since every formula is $\Sigma_n$ for some $n$.

\begin{proof}
Fix any $C(n)$-extendible cardinal $\kappa$, and consider any $\lambda_0 \in C(n)$ above $\kappa$, and an arbitrary $\mathfrak{A}_0=(H_{\lambda_0},\in,\dots)$ in a countable signature.  We need to find some $\mathcal{N}_0 \in \wp^*_\kappa(H_{\lambda_0})$ such that $\mathcal{N}_0 \prec \mathfrak{A}_0$, and the transitive collapse of $\mathcal{N}_0$ is $\Sigma_n$-correct (in the universe).  

Fix a $\lambda \in C(n)$ larger than $\lambda_0$; so in particular, $\mathfrak{A}_0$ is an element of $H_{\lambda}$.  Since the $\Sigma_n$ satisfaction relation is $\Sigma_n$-definable, and since $\mathcal{H}_\lambda \prec_{\Sigma_n} (V,\in)$, it suffices to show that $\mathcal{H}_{\lambda}$ believes there is an $\mathcal{N}_0 \in \wp^*_\kappa(H_{\lambda_0})$ with those properties.  By the $C(n)$-extendibility of $\kappa$, there is some $\lambda' \in C(n)$ and $j: \mathcal{H}_\lambda \to \mathcal{H}_{\lambda'}$ as in the definition of $C(n)$-extendibility.  First, note that since $\lambda_0$ and $\lambda'$ are both in $C(n)$, it follows that
\begin{equation}\label{eq_0_prec_prime}
\mathcal{H}_{\lambda_0} \prec_{\Sigma_n} \mathcal{H}_{\lambda'}.
\end{equation}
Let $Z_0:=j[H_{\lambda_0}]$.  Then $Z_0$ is a bounded subset of $H_{\lambda'} = V_{\lambda'}$, and hence $Z_0 \in \mathcal{H}_{\lambda'}$.  Work inside $\mathcal{H}_{\lambda'}$ for the moment.  Then $Z_0$ is of size strictly smaller than $j(\kappa)$, and $Z_0 \cap j(\kappa) = \kappa \in j(\kappa)$.  Also, the transitive collapse of $Z_0$ is $\mathcal{H}_{\lambda_0}$, which by \eqref{eq_0_prec_prime} is $\Sigma_n$-correct.  Finally, elementarity of $j$ easily yields that $Z_0$ is an elementary substructure of $j(\mathfrak{A}_0)$.  So $\mathcal{H}_{\lambda'} \models$ ``There exists an element of $\wp^*_{j(\kappa)}(j(H_{\lambda_0}))$ that is an elementary substructure of $j(\mathfrak{A}_0)$, and that is $\Sigma_n$-transitively correct".  By elementarity of $j$, $\mathcal{H}_{\lambda}$ believes there is an element of $\wp^*_\kappa(H_{\lambda_0})$ that is elementary in $\mathfrak{A}_0$ and $\Sigma_n$-transitively correct.  
\end{proof}

The corollary below suppresses most of the logical aspects discussed above, which might make it more usable than Lemma \ref{lem_VP_Sigman} or the large cardinal principle \eqref{lc_PhiRef}.  By a \textbf{relation} we mean a (finitary) set-theoretically definable relation, possibly with suppressed parameters.  I.e., $P$ is a relation if $P \subseteq V^n$ for some (meta-mathematical) natural number $n$, and there is a formula $\phi(u_1,\dots,u_n, w_1,\dots,w_k)$ in the language of set theory, and parameters (sets) $p_1,\dots,p_k$, such that 
\begin{equation}\label{eq_P_def}
P = \Big\{ (x_1,\dots,x_n) \ : \ \phi(x_1,\dots,x_n, p_1,\dots,p_k)  \Big\}.
\end{equation}
Partial class functions are defined similarly.

If $b \in \mathcal{N} \prec_{\Sigma_1} (V,\in)$, let $b_{\mathcal{N}}$ denote the image of $b$ under the transitive collapsing map of $\mathcal{N}$.  For an $\mathcal{N} \prec_{\Sigma_1}(V,\in)$ and a relation $P$, let us say that \textbf{$\boldsymbol{\mathcal{N}}$ reflects $\boldsymbol{P}$} if for all $a_1,\dots,a_n \in \mathcal{N}$,
\[
(a_1,\dots,a_n) \in P \ \iff \ \Big( (a_1)_{\mathcal{N}},\dots,(a_k)_{\mathcal{N}}\Big) \in P
\]
(where $n$ is the arity of $P$).  If $F: V^n \to V$ is a partial class function, we say that \textbf{$\boldsymbol{\mathcal{N}}$ reflects $\boldsymbol{F}$} if:
\begin{enumerate}
 \item $\mathcal{N}$ is closed under $F$; and
 \item\label{item_CommutesWithF} for each $x_1,\dots,x_n \in \mathcal{N}$: $F(x_1,\dots,x_n)$ is defined if and only if $F\Big( (x_1)_{\mathcal{N}},\dots, (x_n)_{\mathcal{N}} \Big)$ is defined; and if they are defined, then 
  \[
  F\Big( (x_1)_{\mathcal{N}},\dots, (x_n)_{\mathcal{N}} \Big) = \Big( F(x_1,\dots,x_n) \Big)_{\mathcal{N}}.
  \]
  
\end{enumerate} 
\noindent In other words, $\mathcal{N}$ reflects $F$ if $(-)_{\mathcal{N}}$ commutes with $F$.  Note that closure of $\mathcal{N}$ under $F$ is needed in order for the right side of the equation in requirement \eqref{item_CommutesWithF} to be defined. 

\begin{corollary}\label{lem_VP_Reflect}
Assume Vop\v{e}nka's Principle.  Let $P_1, P_2,\dots,P_\ell$ be a finite list of relations, $F_1,\dots, F_k$ a finite list of class functions, and let $n$ be a natural number.  Then there is a proper class of regular $\kappa$ with the following property: for every set $r$, there is an $\mathcal{N} \in \wp^*_\kappa(V)$ such that $r \in \mathcal{N} \prec_{\Sigma_n}(V,\in)$ and $\mathcal{N}$ reflects each $P_i$ and each $F_j$.

\end{corollary}
\begin{proof}
Without loss of generality we can just deal with the functions $F_1,\dots,F_k$.   Say $F_i$ has arity $m_i$, and for each $i \le k$ let $\phi_i$ be the formula, and $p_{i,1},\dots,p_{i,h_i}$ be the parameters, that define the function $F_i$; i.e., such that for all $x_1,\dots,x_{m_i}$ and all $y$,  
\[
F_i(x_1,\dots,x_{m_i}) = y \text{ if and only if }  \phi_i(y,x_1,\dots,x_{m_i}, p_{i,1},\dots,p_{i,h_i}).
\]

Let $s$ be the (finite) set of all the $p_{i,j}$'s.  Without loss of generality (increasing $n$ if necessary) we can assume that each $\phi_i$ is a $\Sigma_n$ formula, and $n \ge 1$.  Then the formula
\[
\exists w \ \ \phi(w,u_1,\dots, u_{m_i}, v_{i,1},\dots,v_{i,h_i})
\]
is also $\Sigma_n$, and hence any $\Sigma_n$-elementary substructure of the universe that contains $s$ (as a subset) will be closed under each $F_i$.

 Let $\kappa$ be as in the conclusion of Lemma \ref{lem_VP_Sigman} (with respect to $n$), and such that $s \in H_\kappa$; the latter is possible because the lemma gives us a proper class of $\kappa$ with the desired properties.  Fix any set $r$.  By Lemma \ref{lem_VP_Sigman} there is a $\lambda$ such that $r \in \mathcal{H}_\lambda \prec_{\Sigma_n} (V,\in)$, and some $\mathcal{N} \in \wp^*_\kappa(H_\lambda)$ with $\{ s,r \} \subset \mathcal{N} \prec \mathcal{H}_\lambda \prec_{\Sigma_n} (V,\in)$, such that the transitive collapse $\mathcal{H}_{\mathcal{N}}$ of $\mathcal{N}$ is a $\Sigma_n$-elementary substructure of the universe.  Also, $s$ is an element of $\mathcal{N}$, $s$ is a finite subset of $H_\kappa$, and $\mathcal{N} \cap \kappa$ is transitive; it follows from Fact \ref{fact_BasicElemSub} that
\begin{equation*}
s \subset \mathcal{N}, \text{ and the collapsing map of } \mathcal{N} \text{ fixes each member of } s.
\end{equation*}
By the earlier remarks, 
\[
\mathcal{H}_\lambda, \ \mathcal{N}, \text{ and } \mathcal{H}_{\mathcal{N}} \text{ are all closed under each } F_i.  
\]

Consider a fixed $i$; let $F$ denote $F_i$ and drop the other $i$ subscripts for notational simplicity.  Fix $(a_1,\dots,a_m) \in \mathcal{N}$.  Then 
\begin{align*}
& F(a_1,\dots,a_m) \text{ is defined }   \\
\iff & (V,\in)  \models     \ \exists y \  \phi(y,a_1,\dots,a_m, p_1,\dots,p_h) & \text{(definition of $F$)} \\
\iff & \mathcal{H}_\lambda  \models     \ \exists y \  \phi(y,a_1,\dots,a_m, p_1,\dots,p_h) & (\Sigma_n \text{-elementarity of } \mathcal{H}_\lambda ) \\
\iff & \mathcal{N}  \models   \exists y \  \phi(y,a_1,\dots,a_m, p_1,\dots,p_h) & ( \mathcal{N} \prec \mathcal{H}_\lambda ) \\
\iff & \mathcal{H}_{\mathcal{N}}  \models  \ \exists y \   \phi\Big(y, (a_1)_{\mathcal{N}},\dots,(a_m)_{\mathcal{N}}, (p_1)_{\mathcal{N}},\dots,(p_h)_{\mathcal{N}} \Big) & \text{(elementarity of collapsing map)} \\
\iff & \mathcal{H}_{\mathcal{N}}  \models  \ \exists y \   \phi\Big(y, (a_1)_{\mathcal{N}},\dots,(a_m)_{\mathcal{N}}, p_1,\dots,p_h \Big) & \text{(collapsing map fixes members of $s$)} \\
\iff & (V,\in)  \models  \ \exists y \   \phi\Big(y, (a_1)_{\mathcal{N}},\dots,(a_m)_{\mathcal{N}}, p_1,\dots,p_h \Big) & \text{($\Sigma_n$-elementarity $\mathcal{H}_{\mathcal{N}}$)} \\
\iff & F\Big( (a_1)_{\mathcal{N}}, \dots, (a_m)_{\mathcal{N}} \Big) \text{ is defined}  & \text{(definition of $F$).} \\
\end{align*}

\noindent Basically the same argument (using $\phi$ instead of $\exists y \  \phi$) shows that if $F(a_1,\dots,a_m)$ is defined, then the preimage of $F(a_1,\dots,a_m)$ under the collapsing map of $\mathcal{N}$ is the same as $F\Big( (a_1)_{\mathcal{N}}, \dots, (a_m)_{\mathcal{N}}  \Big)$.

\end{proof}

\end{document}